\documentclass[11pt, final]{amsart}
\usepackage{amsmath, amsxtra, amssymb, mathdots, amsthm}
\usepackage{appendix}
\usepackage{verbatim}
\usepackage{multirow}

\usepackage[margin=1.45in]{geometry}
\usepackage{graphicx}
\usepackage[cmtip,all]{xy}

\usepackage[notcite, notref]{showkeys}
\usepackage[colorlinks, linkcolor=blue, citecolor=red, urlcolor=black]{hyperref}

\theoremstyle{plain}
\newtheorem*{main-theorem}{Main Theorem}

\newtheorem{theorem}[equation]{Theorem}
\newtheorem{theorema}{Theorem}
\newtheorem{prop}[equation]{Proposition}
\newtheorem{corollary}[equation]{Corollary}

\newtheorem{claim}[equation]{Claim}
\newtheorem*{claim*}{Claim}

\newtheorem{lemma}[equation]{Lemma}

\theoremstyle{definition}
\newtheorem{definition}[equation]{Definition}
\newtheorem{example}[equation]{Example}
\newtheorem{remark}[equation]{Remark}

\newtheorem{assumption}[equation]{Assumption}

\numberwithin{equation}{section}

\DeclareMathOperator{\spec}{Spec}

\DeclareMathOperator{\Sym}{Sym}

\DeclareMathOperator{\Stab}{Stab}

\DeclareMathOperator{\SL}{SL}

\DeclareMathOperator{\Grass}{Grass}

\DeclareMathOperator{\card}{card}
\DeclareMathOperator{\Res}{Res}

\def\D{\mathcal{D}}

\def\VV{\mathbb{V}}
\def\AA{\mathbf{A}}

\newcommand{\gitq}{/\hspace{-0.25pc}/}

\def\QQ{\mathbb{Q}}
\def\RR{\mathbb{R}}
\def\PP{\mathbb{P}}

\def\ZZ{\mathbb{Z}}
\def\CC{\mathbb{C}}
\def\GG{\mathbb{G}}

\DeclareMathOperator{\chark}{char}

\makeatletter
\def\blfootnote{\xdef\@thefnmark{}\@footnotetext}
\makeatother

\address[Fedorchuk]{Department of Mathematics\\
Boston College\\
140 Commonwealth Ave\\
Chestnut Hill, MA 02467, USA}
\email{maksym.fedorchuk@bc.edu}

\begin{document}

\title[Decomposability of polynomials and associated forms]{Direct sum decomposability of polynomials and 
factorization of associated forms}

\author{Maksym Fedorchuk}

\begin{abstract} We prove two criteria for direct sum decomposability of 
homogeneous polynomials. 
For a homogeneous polynomial with a non-zero discriminant, we
interpret direct sum decomposability of the polynomial in terms
of factorization properties of the Macaulay inverse system of its Milnor algebra.
This leads to an if-and-only-if criterion for direct sum decomposability of such 
a polynomial, and 
to an algorithm for computing direct sum decompositions over any field, 
either of characteristic $0$ or of sufficiently large positive characteristic, 
for which polynomial factorization algorithms exist. 

For homogeneous forms over algebraically closed fields, we interpret
direct sums and their limits as forms that cannot be reconstructed from their
Jacobian ideal. 

We also give 
simple necessary criteria for direct sum decomposability of 
arbitrary homogeneous polynomials over arbitrary fields 
and apply them to prove that many interesting classes of homogeneous
polynomials are not direct sums.
\end{abstract}

\maketitle

\section{Introduction}

A homogeneous polynomial $f$ is called a \emph{direct sum} 
if, after a change of variables, it can be 
written as a sum of two or more polynomials in 
disjoint sets of variables: 
\begin{equation}\label{E:ds}
f=f_1(x_1,\dots, x_a)+f_2(x_{a+1},\dots, x_n).
\end{equation}
Homogeneous direct sums are the subject 
of a well-known symmetric Strassen's additivity conjecture
postulating that the Waring rank of $f$ in \eqref{E:ds} is the sum of the Waring ranks of $f_1$ and $f_2$
(see, for example, \cite{teitler}). Direct sums also play a special role in the study of GIT stability 
of associated forms \cite{fedorchuk-isaev}.

The innocuous definition of a direct sum raises several natural questions:  How do we determine whether
a given polynomial is a direct sum? For example, is 
{
\begin{multline*}
f={x}_{1}^{3}+3 {x}_{1}^{2} {x}_{2}+3 {x}_{1} {x}_{2}^{2}+2 {x}_{2}^{3}+3
      {x}_{1}^{2} {x}_{3}+6 {x}_{1} {x}_{2} {x}_{3}+4 {x}_{2}^{2} {x}_{3}+3
      {x}_{1} {x}_{3}^{2}+4 {x}_{2} {x}_{3}^{2}+2 {x}_{3}^{3}
      \end{multline*}
      }
a direct sum in $\QQ[x_1,x_2,x_3]$? (See Example \ref{E:intro} for the answer).
Does the locus of direct sums in the space of all homogeneous polynomials of a given 
degree has a geometric meaning? 

In this paper, we answer these questions by considering two natural maps on 
the space of forms\footnote{Hereinafter, we refer 
to any homogeneous polynomial as a form, and we
call a form $f$ in $n$ variables smooth if it defines a smooth hypersurface
in $\PP^{n-1}$; see \S\ref{S:notation} for further terminology and notational conventions.} of a given degree, the gradient morphism $\nabla$
and the associated form morphism $A$, described in more detail later on. Our first main result is a new criterion for recognizing when a \emph{smooth} form
 is a direct sum over a field either of characteristic $0$ or of sufficiently large positive characteristic:
 
 \begin{theorema}[{see Theorem \ref{MT1}}] \label{theorem-A} A smooth form $f$ is a direct sum if and only if its associated form $A(f)$ is a nontrivial product of two factors in disjoint variables.  
\end{theorema}
Our second main result is the characterization of 
the locus of direct sums as the non-injectivity locus of $\nabla$, which, as will be clear, is the locus 
where $\nabla$ has positive fiber dimension.  While the full statement of this result given by Theorem \ref{T:injectivity-direct-sum} 
is too cumbersome to state in the introduction, it is well illustrated by the following:
\begin{theorema}[{see Theorem \ref{T:injectivity-direct-sum}}]\label{theorem-B} Suppose $f$ is a GIT semistable
form over algebraically closed field of characteristic not dividing $\deg(f)!$.
Then $f$ is a direct sum if and only if there exists $g$, which is not
a scalar multiple of $f$, and such that $\nabla f=\nabla g$.   
\end{theorema}

The problem of finding a direct sum decomposability 
criterion for an arbitrary homogeneous polynomials has been successfully addressed earlier
by Kleppe \cite{kleppe-thesis} over an arbitrary field, 
and Buczy\'nska-Buczy\'nski-Kleppe-Teitler \cite{apolarity} over an algebraically closed field. 
Both works interpret direct sum decomposability of a form 
$f$ in terms of its apolar ideal $f^{\perp}$ (see \S \ref{S:prior-work} for more details).
Kleppe uses the quadratic part of the apolar ideal $f^{\perp}$ to 
define an associative algebra $M(f)$ of finite dimension over the base field.
%($M(f)$ is different from the Milnor algebra $M_f$). 
He
then proves that, over an arbitrary field, direct sum decompositions of $f$ are 
in bijection with complete sets of orthogonal idempotents
of $M(f)$.  Buczy\'nska, Buczy\'nski, Kleppe, and Teitler prove
that for a form $f$ of degree $d$ over an algebraically closed field,
the apolar ideal $f^{\perp}$ has 
a minimal generator in degree $d$ if and only if either $f$ is a direct sum, or $f$ is 
a highly singular polynomial. 
In particular, over an algebraically closed field, \cite{apolarity} gives an effective criterion for 
recognizing when $f$ is a direct sum in terms of the graded Betti numbers of $f^{\perp}$. 
We will show that Theorem \ref{theorem-B} is essentially equivalent to the main result of \cite{apolarity},
thus giving a different proof of this result. 

However, none of the above-mentioned two works seem to give an effective
method for computing a direct sum decomposition when it exists, and the criterion of
\cite{apolarity} cannot be used over non-closed fields (see Example \ref{E:non-closed}).
A key step in the proof of the direct sum criterion in \cite{apolarity} is the
Jordan normal form decomposition of a certain linear operator, which in general requires solving a
characteristic equation. 
Similarly, finding a complete set of orthogonal idempotents requires solving a system of quadratic equations. This makes it challenging to turn \cite{apolarity} or \cite{kleppe-thesis} into an algorithm for finding
direct sum decompositions when they exist.  
%Our criterion gives an algorithm for computing direct sum decompositions 
%and works over non-closed field. 
Although our %if-and-only-if 
criterion given by Theorem \ref{theorem-A} works only for smooth forms, it does so over an arbitrary field either of characteristic $0$ or of 
sufficiently large characteristic, and it leads
to an algorithm for finding direct sum decompositions over any such field for which polynomial factorization
algorithms exist. This algorithm is given in Section \ref{S:algorithm}. 

Recall that to a smooth form $f$ of degree $d+1$ in $n$ variables, 
one can assign a degree $n(d-1)$ form $A(f)$ in $n$ (dual) variables, 
called the \emph{associated form} of $f$
(\cite{alper-isaev-assoc,alper-isaev-assoc-binary,eastwood-isaev1}).
The associated form $A(f)$ is defined as a Macaulay inverse system of the Milnor algebra of $f$ 
\cite{alper-isaev-assoc}, which simply means that the apolar ideal of $A(f)$ coincides with the Jacobian ideal of $f$:
\[
A(f)^{\perp}=(\partial f/\partial x_1, \dots, \partial f/\partial x_n).
\]
Such definition leads to an
observation that for a smooth %direct sum decomposable 
form $f$ that is written as a \emph{sum} of two forms in disjoint
sets of variables, the associated form $A(f)$ 
decomposes as a \emph{product} of two forms in disjoint sets of (dual) variables (\cite[Lemma 2.11]{fedorchuk-isaev}). 
For example, up to a scalar,
\[
A(x_1^{d+1}+\cdots+x_{n}^{d+1})=z_1^{d-1}\cdots z_n^{d-1}.
\]
The main purpose of Theorem \ref{theorem-A} is to prove the converse
statement, and thus establish an if-and-only-if criterion for direct sum decomposability of a smooth form $f$
in terms of the factorization properties of its associated form $A(f)$ (see Theorem \ref{MT1}). % for an extended version):
%\begin{main-theorem} 
%\end{main-theorem} 

In Lemma \ref{L:gradient}, we give a simple necessary condition, valid over an arbitrary field, 
for direct sum decomposability of 
an arbitrary form in terms of its \emph{gradient point}.
It is then applied in Section \ref{S:necessary} to prove that a wide class of homogeneous forms contains no direct sums.
%In Theorem \ref{MT2}, we show that this simple necessary condition is in fact sufficient when 
%a form is GIT stable over an algebraically closed field. 

\subsection{Notation and conventions}
\label{S:notation}
Throughout, $k$ will be a field. Unless stated otherwise, we do not require $k$ to be algebraically 
closed or to be of characteristic $0$. 

If $U$ is a $k$-vector space, and $W$ is a subset of $U$, we denote by $\langle W \rangle$ the $k$-linear span of $W$ in $U$.

If $W$ is a representation of the multiplicative group scheme $\GG_m=\spec k[t,t^{-1}]$, 
then, for every $i\in \ZZ$, we denote by $W_{(i)}$
the weight-space in $W$ of the $\GG_m$-action of weight $i$.  We then have a decomposition 
\[
W=\bigoplus_{i\in \ZZ} W_{(i)}.
\]

We fix an integer $n\geq 2$ and fix a $k$-vector space $V$ of dimension $n$. % $\dim_k V = n$. 
We set $S:=\Sym V$, and $\D:=\Sym V^{\vee}$ to be the symmetric algebras on $V$ 
and $V^{\vee}$, respectively, with the standard grading.  If $x_1,\dots, x_n$ is a basis of $V$,
then we identify $S$ with the usual polynomial ring $k[x_1,\dots, x_n]$. If $z_1,\dots, z_n$
is the dual basis of $V^{\vee}$, then $\D=k[z_1,\dots, z_n]$ is the graded dual of $S$. 
Homogeneous elements of $S$ and $\D$ will be called \emph{forms}. 
If $f\neq 0\in S_{d}$ is a degree $d$ form, then we will denote
by $\bar f$ the corresponding element of the projective space $\PP S_d$. 

For a homogeneous ideal $I\subset S$, we denote by $\VV(I)$ the closed subscheme of 
$\PP V^{\vee}\simeq \PP^{n-1}$ defined by $I$. 
In particular, 
a nonzero form $f\in S_d$ defines a hypersurface $\VV(f)$ in $\PP V^{\vee}$. 
If $k$ is algebraically closed,
the hypersurface $\VV(f)$ determines $\bar f$, and so we will sometimes not distinguish 
between a hypersurface and its equation. 
With this in mind, we say that a form $f\in S_{d+1}$ is \emph{smooth} 
if the hypersurface $\VV(f)\subset \PP^{n-1}$ is smooth over $k$ 
(this is, of course, equivalent to $\VV(f)$ being non-singular 
over the algebraic closure of $k$). The locus of smooth forms
in $\PP S_{d+1}$ will be denoted by $(\PP S_{d+1})_{\Delta}$.

We have a differentiation action of 
$S$ on $\D$, also known as \emph{the polar pairing}. Namely, if $x_1,\dots, x_n$ is a basis of $V$, and $z_1,\dots,z_n$
is the dual basis of $V^{\vee}$, then the pairing 
$
S \times \D \to \D
%S_{d'} \times \D_{d} \to \D_{d-d'} 
$
is given by
\[
g\circ F:=g(\partial /\partial z_1, \dots, \partial /\partial z_n) F(z_1,\dots,z_n), \quad \text{for $g\in S$ and $F\in \D$.}
\]

Given a non-zero form $F\in \D_{d}$, 
 \emph{the apolar ideal of $F$} is defined to be
\begin{equation*}%\label{E:apolar-ideal}
F^{\perp}:=\{g\in S \mid g\circ F=0\}\subset S.
\end{equation*}
We let \emph{the space of essential variables} of $F$ to be
\begin{equation*}%\label{E:essential}
E(F):=\{g\circ F \mid g\in S_{d-1}\} \subset \D_1.
\end{equation*}

If $\chark(k)\nmid d!$, then the pairing $S_{d} \times \D_{d} \to k$ is perfect, and 
for every $F\in \D_{d}$, we have
$F\in \Sym^{d} E(F) \subset \D_{d}$ (see \cite[\S2.2]{apolarity} for details), so that
$F$ can be expressed as a polynomial in its essential variables. This property 
explains the name of $E(F)$. 

Importantly, working under the assumption that $\chark(k)\nmid d!$, the graded $k$-algebra $S/F^{\perp}$ is a Gorenstein Artin local ring 
with socle in degree $d$. In fact, a 
well-known theorem of Macaulay establishes a bijection between graded Gorenstein 
Artin quotients $S/I$ of socle degree
$d$ and elements of $\PP \D_{d}$ 
(see, e.g., \cite[Lemma 2.12]{iarrobino-kanev} or \cite[Exercise 21.7]{eisenbud}).

\begin{definition}[Gradient morphism]
\label{D:gradient}
Let $x_1,\dots,x_n$ be a basis of $V$. For $f\in S_{d+1}$,
we define
\[
\langle\nabla f\rangle:=\langle \partial f/\partial x_1,\dots, \partial f/\partial x_n\rangle \subset S_{d}
\]
to be the span of all first-order partials of $f$.  
If $\dim_{k} \langle\nabla f\rangle=n$, we denote by $\nabla f$ the point of $\Grass(n, S_d)$
corresponding to $\langle\nabla f\rangle$, and call $\nabla f$ \emph{the gradient point of $f$}.
\end{definition}

The \emph{Jacobian ideal} of $f$ is defined to be \[
J_f:=(\nabla f)=(\partial f/\partial x_1,\dots, \partial f/\partial x_n) \subset S,
\]
and the \emph{Milnor algebra} of $f$ is $M_f:=S/J_f$.

Following the terminology of \cite[\S2.2]{apolarity}, 
we call $f\in S_{d}$ (respectively, $\bar f \in \PP S_{d}$) \emph{concise} if 
it cannot be written as a form in less than $n=\dim_k V$ variables, or 
equivalently if $f\in \Sym^{d} W$ for $W\subset V$ implies that $W=V$. 
If $\chark(k) \nmid d!$, then $f\in S_d$ is concise if and only if 
$E(F)=V$ if and only if $\dim_{k}\langle \nabla f\rangle =n$.\footnote{But note
that $f=x^2+yz$ is always concise in $k[x,y,z]$ while $\dim_k \langle \nabla f\rangle=2<3$
if $\chark(k)=2$.}
We set 
\[
\PP(S_{d})^{c}:=\{\bar f \in  \PP S_{d} \mid \dim_{k}\langle \nabla f\rangle =n\}.
\]
Clearly, $\PP(S_{d})^{c}$ is an open subset of $\PP S_{d}$.
If $\chark(k) \nmid d!$, then $\PP(S_{d})^{c}$ is the locus of concise forms, and so 
$\bar f\in \PP S_{d}\setminus \PP(S_{d})^{c}$ if and only
if the hypersurface defined by $f$ is a cone if and only if $E(f)\subsetneq V$.

\begin{remark} Even though we allow $k$ to have positive characteristic, we do not take $\D$ to be
the divided power algebra (cf. \cite[Appendix A]{iarrobino-kanev}), as the reader might have anticipated. 
The reason for this is that at several places
we cannot avoid but to impose a condition that $\chark(k)$ is large enough (or zero).  
In this case, the divided power algebra is isomorphic to $\D$
up to the degree in which we work.
\end{remark}

\subsection{Direct sums and products}  
In this subsection, we recall a well-known definition of a direct sum polynomial, and 
introduce more exotic notions of a direct product polynomial, and a direct sum space of polynomials.

\subsubsection{Direct sum forms}
A form $f\in \Sym^{d+1} V$ is called a \emph{direct sum}
if there is a direct sum decomposition $V=U\oplus W$ and nonzero $f_1\in \Sym^{d+1} U$
and $f_2\in \Sym^{d+1} W$ such that $f=f_1+f_2$.
In other words, $f$ is a direct sum
if and only if for some choice of a basis $x_1,\dots, x_n$ of $V$, we have that
\[
f=f_1(x_1,\dots, x_a)+f_2(x_{a+1},\dots,x_n), 
\]
where $1\leq a\leq n-1$, and $f_1, f_2 \neq 0$.

\begin{remark} Note that the roles of $S$ and $\D$ are interchangeable in \S\ref{S:notation}, and so for $f\in S$,
we can define the apolar ideal $f^{\perp} \subset \D$ and the space of essential variables
$E(f)\subset S_1$. With this notation, if $\chark(k)\nmid (d+1)!$, then 
$f\in S_{d+1}$ is a direct sum if and only if we
can write $f=f_1+f_2$, where $f_1,f_2\neq 0$ and $E(f_1)\cap E(f_2)=0$.
Furthermore, under the same assumption on $\chark(k)$,
we have that %$E(f)\subset V$ is dual to $(f^{\perp})_1\subset V^{\vee}$, and 
$\dim_k \langle \nabla f\rangle=\dim_k E(f)$.
\end{remark}

\subsubsection{LDS forms} 
We say that $f\in S_{d+1}$ is an \emph{LDS} form (or, simply, $f$ is LDS) 
if, after a linear change of variables, 
we can write
\begin{equation}\label{E:lds} 
f(x_1,\dots,x_n)=\sum_{i=1}^\ell x_i \frac{\partial H(x_{\ell+1}, \dots, x_{2\ell})}{\partial x_{\ell+i}} + G(x_{\ell+1},\dots, x_n),
\end{equation} 
where $H$ and $G$ are degree $d+1$ forms, in $\ell$ and $n-\ell$ variables, respectively.
To our knowledge, LDS forms first appeared in \cite[\S 4.2]{apolarity} as forms
that are limits of direct sums. We recall that by \cite[Equation (4), p.697]{apolarity},
we have that $f$ from \eqref{E:lds} satisfies
\begin{multline}
f=\lim_{t\to 0}\frac{1}{t}\bigl[H(tx_{1}+x_{\ell+1}, \dots, tx_{\ell}+x_{2\ell})-H(x_{\ell+1},\dots,x_{2\ell})\\
+tG(tx_{1}+x_{\ell+1}, \dots, tx_{\ell}+x_{2\ell}, x_{2\ell+1},\dots, x_n)\bigr].
\end{multline}
Since the form on the right-hand side of the above equation is a direct sum for $t\neq 0$, 
we obtain an explicit presentation of the LDS form $f$ as a limit of direct sums. 
Conversely, \cite[Theorem 4.5]{apolarity} proves
that every concise limit of direct sums that is not itself a direct sum is an LDS form. 

\begin{lemma}\label{L:lds} If $f$ is an LDS form of degree $\deg(f)=d+1\geq 3$, then $f$ is GIT non-semistable
with respect to the standard $\SL(n)$-action on $S_{d+1}$.
\end{lemma}

\begin{proof}
Suppose $f$ is as in \eqref{E:lds}.
%\begin{equation}\label{quasi-deg-2} 
%f(x_1,\dots,x_n)=\sum_{i=1}^\ell x_i \frac{\partial H(x_{\ell+1}, \dots, x_{2\ell})}{\partial x_{\ell+i}} + G(x_{\ell+1},\dots, x_n).
%\end{equation} 
Then with respect to the one-parameter subgroup of $\SL(n)$ acting on a basis $x_1,\dots, x_{\ell}, x_{\ell+1}, \dots, x_{2\ell}, x_{2\ell+1}, \dots, x_n$ of $V$ with weights
\[
(1+\epsilon, \underbrace{1,\dots, 1}_{\ell-1}, \underbrace{-1,\dots, -1}_{\ell},  -\epsilon/(n-2\ell), \dots, -\epsilon/(n-2\ell)), \quad \text{where $0<\epsilon \ll 1$},\] 
the maximum weight of a monomial in $f$ is 
\[
\max\{(1+\epsilon)-d, -(d+1)\epsilon/(n-2\ell)\}<0.
\]
Hence $f$ is non-semistable by the Hilbert-Mumford numerical criterion.
\end{proof}

\subsubsection{Direct product forms}
By analogy with direct sums, %with the above definition, 
we will call a nonzero form $F \in \D$ a \emph{direct product}
if there is a non-trivial direct sum decomposition $V^{\vee}=U\oplus W$ such that 
$F=F_1F_2$ for some $F_1\in \Sym U$ and $F_2\in \Sym W$. 
In other words, a nonzero homogeneous $F\in \Sym \D$ is a direct product
if and only if for some choice of a basis $z_1,\dots, z_n$ of $V^{\vee}$, we have that
\begin{equation}\label{E:direct-product}
F(z_1,\dots,z_n)=F_1(z_1,\dots, z_a)F_2(z_{a+1},\dots,z_n), \quad \text{where $1\leq a\leq n-1$.}
\end{equation} 
Furthermore, we call a direct product decomposition in \eqref{E:direct-product} \emph{balanced} if 
\[
(n-a)\deg(F_1)  = a\deg (F_2).
\]
Note that if $\chark(k)\nmid (\deg F)!$, then a non-trivial factorization 
$F=F_1F_2$ is a direct product decomposition if and only if 
$E(F_1)\cap E(F_2)=0$. This observation reduces the problem of recognition of direct products
to a factorization problem.

\subsubsection{Direct sum and decomposable spaces of forms}
\label{direct-sum-spaces}
Suppose $L\subset \Sym^d V$ is a subspace of dimension $m$.
We say that $[L]\in \Grass(m, \Sym^d V)$ %, where as always $n=\dim_k V$, 
is a \emph{direct sum} 
if there is a non-trivial direct sum decomposition $V=U\oplus W$ and elements
$[L_1]\in \Grass(m_1, \Sym^{d} U)$ and $[L_2]\in \Grass(m_2, \Sym^{d} W)$, where $m_1+m_2=m$, such that
\begin{equation}\label{E:bds}
L=L_1+L_2 \subset \Sym^{d} U\oplus \Sym^{d} W \subset \Sym^{d} V.
\end{equation}
%When $n=\dim V$, w
We will further say that $[L]$ is a \emph{balanced direct sum} 
if $m_1=\dim_k U$ and $m_2=\dim_k W$.
Note that in this case necessarily $m=\dim_k V=n$. 
\begin{remark}\label{R:Gm-bds} Direct sums are stabilized by 
non-trivial one-parameter subgroups of $\SL(V)$. For example, the space $L$ in \eqref{E:bds} is fixed by any one-parameter subgroup acting with weight $\dim_k W$ on $U$ and weight $-\dim_k U$ on $W$. 
\end{remark}

Finally, we say that $[L]\in \Grass(n, \Sym^d V)$ is \emph{decomposable}
if there is a non-trivial subspace $U\subset V$ and elements
$[L_1]\in \Grass(\dim_k U, \Sym^{d} U)$ and $[L_2]\in \Grass(n-\dim_{k} U, \Sym^{d} V)$ such that
\[
L=L_1+L_2 \subset \Sym^{d} V.
\]

\subsection{Associated forms} Next, we briefly recall the theory of associated forms as developed in 
\cite{eastwood-isaev1,eastwood-isaev2,alper-isaev-assoc-binary,fedorchuk-isaev}. To avoid
trivialities, we adopt the following:
\begin{assumption}\label{ass-nd} Assume $n\geq 2$, and that $d\geq 3$ if $n=2$,
and $d\geq 2$ otherwise. In particular, we have in this case that $n(d-1)\geq d+1$.
\end{assumption}

Let $\Grass(n,S_d)_{\Res}$ be the affine open subset %of complete
%intersections  
in $\Grass(n,S_d)$ %Namely, $\Grass(n,S_d)_{\Res}$ 
parameterizing linear 
subspaces $\langle g_1,\dots, g_n\rangle \subset S_d$ such that $g_1,\dots, g_n$ form a regular 
sequence in $S$, or, equivalently, such that the ideal $(g_1,\dots, g_n)$ is a complete intersection,
or, equivalently, such that the resultant $\Res(g_1,\dots, g_n)$ is nonzero.
Note that, if $\chark(k)\nmid (d+1)!$, then $f\in S_{d+1}$ 
is smooth if and only if $\nabla f \in \Grass(n,S_d)_{\Res}$.

For every $U=\langle g_1,\dots, g_n\rangle  \in \Grass(n,S_d)_{\Res}$, the ideal
$I_U=(g_1,\dots, g_n)$ is a complete intersection ideal, and the $k$-algebra
$S/I_U$ is a graded Gorenstein Artin local ring with socle in degree $n(d-1)$.
Suppose $\chark(k)\nmid (n(d-1))!$.
Then, by Macaulay's theorem, there exists a unique up to scaling form $\AA(U) \in \D_{n(d-1)}$ such that
\begin{equation*}%\label{associated-form-U}
\AA(U)^{\perp} = I_U.
\end{equation*}
The form $\AA(U)$ is called \emph{the associated form of $g_1,\dots,g_n$} by
Alper and Isaev, who systematically studied it in \cite[Section 2]{alper-isaev-assoc-binary}.
In particular, they showed\footnote{
Although given over $\CC$, their proof
applies whenever $\chark(k)=0$ or $\chark(k)>n(d-1)$.} that the assignment 
$U \to \overline{\AA(U)}\in \PP \D_{n(d-1)}$ 
gives rise to an $\SL(n)$-equivariant 
\emph{associated form morphism}
\[\label{E:assoc-morphism}
\AA\colon \Grass(n,S_d)_{\Res}\rightarrow \PP \D_{n(d-1)}. %\qquad \langle g_1,\dots,g_n\rangle\mapsto \AA(g_1,\dots,g_n).
\]
When $U=\nabla f$ for a smooth form $f\in S_{d+1}$, we set 
\[
A(f):=\AA(\nabla f)=\AA(\langle \partial f/\partial x_1, \dots, \partial f/\partial x_n\rangle),
\] and, following 
Eastwood and Isaev \cite{eastwood-isaev1}, call $A(f)$
\emph{the associated form of $f$}.
The defining property of $A(f)$ is that 
the apolar ideal of $A(f)$ is the Jacobian ideal of $f$:
\begin{equation*}%\label{E:AF}
A(f)^{\perp}=J_f.
\end{equation*}
This means that $A(f)$ is a homogeneous 
Macaulay inverse system of the Milnor algebra $M_f=S/J_f$.

Summarizing, when $\chark(k)\nmid (n(d-1))!$, we have the following commutative diagram of $\SL(n)$-equivariant morphisms:

\begin{equation*}
\xymatrix{ 
\bigl(\PP S_{d+1} \bigr)_{\Delta}  \ar[rr]^{\nabla} \ar[rd]^{A}& & \Grass\bigl(n, S_d\bigr)_{\Res}\ar[dl]^{\AA} \\
& \PP\bigl(\D_{n(d-1)}\bigr)&
}
\end{equation*}

\begin{remark} In \cite{alper-isaev-assoc-binary}, Alper and Isaev define the associated form 
$\AA(g_1,\dots,g_n)$ as an element
of $\D_{n(d-1)}$, which they achieve by choosing a canonical generator of the socle of $S/(g_1,\dots,g_n)$ given
by the Jacobian determinant of $g_1,\dots, g_n$. For our purposes, it will suffice to 
consider $\AA(\langle g_1,\dots,g_n\rangle)$ defined up to a scalar. 
\end{remark}

\section{First properties of direct sums and LDS forms}

We begin by recording several immediate properties of direct sums and, more generally, of LDS forms.

\begin{lemma}\label{L:gradient-ds} Suppose $f\in S_{d+1}$ is a direct sum. Let $\bar f$ be its image in $\PP S_{d+1}$. Then the following hold:
\begin{enumerate}
\item\label{L:grad-gm} 
If $k\neq \mathbb{F}_2$, there is a non-trivial 
one-parameter subgroup $\rho \colon \GG_m \hookrightarrow \SL(V)$ defined over $k$ 
such that $\rho\cdot \langle \nabla f\rangle=\langle \nabla f\rangle$ but
$\rho\cdot \overline{f}\neq \overline{f}$.  Consequently, $\Stab_{\SL(V)}(\bar f)$ is
a proper subgroup of $\Stab_{\SL(V)}(\langle \nabla f\rangle)$.
%Moreover,
%we have the following $\rho$-weight-space decompositions 
%\[%V=V_{\left(\frac{n-a}{\gcd(n-a,a)}\right)}\oplus V_{\left(-\frac{a}{\gcd(n-a,a)}\right)} \ \text{and} \
%\langle \nabla f\rangle=\langle\nabla f\rangle_{\left(\frac{(n-a)d}{\gcd(n-a,a)}\right)}\oplus \langle \nabla f\rangle_{\left(\frac{-ad}{\gcd(n-a,a)}\right)}.
%\] 
%\item $\langle \nabla f\rangle \subspace S_d$ is a direct sum. POSITIVE CHARACTERISTIC!
\item The set of $k$-points of \[
\{\bar g\in \PP S_{d+1} \mid \langle \nabla g\rangle=\langle\nabla f\rangle\}
\]
 contains $k^*=k\setminus\{0\}$. 
\item\label{L:grad-bds} If furthermore $\dim_k \langle \nabla f\rangle=n$, then
$\nabla f \in \Grass(n, S_d)$ is a balanced direct sum (cf. \S\ref{direct-sum-spaces}),
and the set of $k$-points of \[
\{\bar g\in \PP S_{d+1} \mid \langle \nabla g\rangle=\langle\nabla f\rangle\ \text{and $\langle\nabla (f-cg)\rangle\neq 0$ for all $c\in k$}\}
\]
contains $k\setminus \{0,1\}$.
\end{enumerate}
\end{lemma}
\begin{proof} Suppose \[
f=f_1(x_1,\dots,x_a)+f_2(x_{a+1},\dots, x_n)\]
in some basis $x_1,\dots, x_n$ of $V$, where $f_1,f_2\neq 0$. 
Let $\rho$ be the one-parameter subgroup of $\SL(V)$ 
acting with weight $(n-a)/\gcd(n-a,a)$ on the variables $\{x_i\}_{i=1}^a$
and weight $-a/\gcd(n-a,a)$ on $\{x_i\}_{i=a+1}^n$. Since the weights are of opposite sign
and coprime, $\rho$ is a non-trivial subgroup of $\SL(V)$ defined over $k$. 
Then we have
\begin{multline}\label{E:nabla-direct}
\langle \nabla f\rangle=\langle \nabla f_1\rangle\oplus \langle \nabla f_2\rangle,  \\
 \text{where $\langle\nabla f_1\rangle \subset k[x_1,\dots, x_a]$ and $\langle\nabla f_2\rangle \subset k[x_{a+1},\dots, x_n]$.}
\end{multline}

It is clear that $\langle\nabla f\rangle$
is invariant under $\rho$, and that $f$ is not $\rho$-invariant because
$f_1, f_2\neq 0$. This proves (1).

Clearly, the forms \[
f_{\lambda}:=f_1+\lambda f_2,
\] where $\lambda \in k^*$, all satisfy 
\[
\langle \nabla f_{\lambda}\rangle=\langle \nabla f\rangle,
\] and are pairwise non-proportional for distinct values of $\lambda$. This proves (2).

Suppose further that $\dim_k\langle \nabla f\rangle=n$ so that $f$ is concise.
Then in \eqref{E:nabla-direct}, we have necessarily 
that $\dim_k\langle \nabla f_1\rangle=a$ 
and $\dim_k\langle \nabla f_2\rangle=n-a$.
Thus $\nabla f$ is a balanced direct sum. We also see that 
\[
\nabla (f-c(f_1+\lambda f_2))=(1-c)\nabla f_1+(c-\lambda) \nabla f_2 \neq 0
\]
for all $\lambda\neq 1$. 
This proves (3).
\end{proof}

Under further assumptions on the characteristic of the field $k$, the converse to 
Lemma \ref{L:gradient-ds}(3) holds, and we have the following simple characterization of
direct sums in terms of their gradient points. 
\begin{lemma}\label{L:gradient}
%Assume $\dim_{k} \langle \nabla f\rangle=n$. (This is for example
%satisfied whenever 
Assume $\chark(k) \nmid (d+1)!$ and 
$f\in S_{d+1}$ is a concise form (equivalently, $\dim_k \langle \nabla f\rangle=\dim_k V=n$). 
Then $f$ is a direct sum if and only if 
$\nabla f \in \Grass(n,S_d)$ is a direct sum if and only if 
$\nabla f \in \Grass(n,S_d)$ is a balanced direct sum.  
\end{lemma}
\begin{proof}
Suppose $f=f_1(x_1,\dots,x_a)+f_2(x_{a+1},\dots, x_n)$ 
in some basis $x_1,\dots, x_n$ of $V$, where $f_1,f_2\neq 0$. 
Then using $\dim_{k} \langle \nabla f\rangle=n$, we deduce from \eqref{E:nabla-direct}
that $\langle \nabla f\rangle$ is a balanced direct
sum.

Suppose $\nabla f$ decomposes as a %balanced  
direct sum in a basis $x_1,\dots,x_n$ of $V$. Then
\[
\nabla f=\langle s_1,\dots, s_b, t_1, \dots, t_{n-b}\rangle,
\]
for some $s_1,\dots,s_b\in k[x_1,\dots,x_a]_d$ and $t_1, \dots, t_{n-b}\in k[x_{a+1},\dots,x_n]_d$.

Then for every $1\leq i \leq a$ and $a+1\leq j \leq n$, we can write 
\begin{align*}
\partial f/\partial x_i &= u_1+v_1,  \\
\partial f/\partial x_j &= u_2+v_2,
\end{align*}
where $u_1, u_2 \in k[x_1,\dots,x_a]_d$ and $v_1, v_2 \in k[x_{a+1},\dots,x_n]_d$.
It follows that 
\[
\frac{\partial^2 f}{\partial x_i \partial x_j} =\frac{\partial v_1}{\partial x_j} \in k[x_{a+1},\dots,x_n]_{d-1}.
\]
and 
\[
\frac{\partial^2 f}{\partial x_j \partial x_i} =\frac{\partial u_2}{\partial x_i} \in k[x_1,\dots,x_a]_{d-1}.
\] 
In other words,  for every $1\leq i \leq a$ and $a+1\leq j \leq n$, we have
\[
\frac{\partial^2 f}{\partial x_i \partial x_j} \in k[x_1,\dots,x_a]_{d-1}\cap k[x_{a+1},\dots,x_n]_{d-1} = (0).
\]
It follows that $\dfrac{\partial^2 f}{\partial x_i \partial x_j}=0$ for all $1\leq i \leq a$ and $a+1\leq j \leq n$. 
Using the assumption on $\chark(k)$, we conclude that 
\[
f \in k[x_1,\dots, x_a]_{d+1} \oplus k[x_{a+1},\dots,x_n]_{d+1},
\]
and so $f$ is a direct sum, in the same basis as $\nabla f$.

\end{proof}

For LDS forms, we have the following analog of Lemma \ref{L:gradient-ds}.
\begin{lemma}\label{L:gradient-lds} Suppose $f\in S_{d+1}$ is an LDS form.
Let $\bar f$ be its image in $\PP S_{d+1}$. 
Then the following hold:
\begin{enumerate}
\item\label{L:grad-gm} 
If $k\neq \mathbb{F}_2$, there is a non-trivial 
one-parameter subgroup $\rho \colon \GG_m \hookrightarrow \SL(V)$ 
such that $\rho\cdot \langle \nabla f\rangle=\langle \nabla f\rangle$ but
$\rho\cdot \overline{f}\neq \overline{f}$.  Consequently, $\Stab_{\SL(V)}(\bar f)$ is
a proper subgroup of $\Stab_{\SL(V)}(\langle \nabla f\rangle)$.
%Moreover,
%we have the following $\rho$-weight-space decompositions 
%\[V=V_{(n-a)}\oplus V_{(-a)} \ \text{and} \
%\langle \nabla f\rangle=\langle\nabla f\rangle_{((n-a)d)}\oplus \langle \nabla f\rangle_{(-ad)}.
%\] 
\item The set of $k$-points of
 \[
 \{\bar g\in \PP S_{d+1} \mid \langle\nabla g\rangle=\langle\nabla f\rangle\}
 \]
  contains $k^*$. 

\item\label{L:grad-bds} If furthermore $\dim_k \langle \nabla f\rangle=n$, then 
$\nabla f \in \Grass(n, S_d)$ is decomposable. 
\end{enumerate}
\end{lemma}

\begin{proof} Suppose
\begin{equation}
f=\sum_{i=1}^a x_i \frac{\partial H(x_{a+1}, \dots, x_{2a})}{\partial x_{a+i}} + G(x_{a+1},\dots, x_n),
\end{equation} 
where $H$ and $G$ are non-zero degree $d+1$ forms, in $a$ and $n-a$ variables, respectively.
Let $\rho$ be the one-parameter subgroup of $\SL(V)$ 
acting with weight $(n-a)/\gcd(n-a,a)$ on the variables $\{x_i\}_{i=1}^a$
and weight $-a/\gcd(n-a,a)$ on $\{x_i\}_{i=a+1}^n$. Since the weights are of opposite sign
and coprime, $\rho$ is a non-trivial subgroup of $\SL(V)$ defined over $k$. 
The proof of (1) and (2) now proceeds as in the proof of 
Lemma \ref{L:gradient-ds} by noting that $\langle \nabla f\rangle$ does 
not change if we multiply $H$ by an element of $k^*$.

Finally (3) follows from the fact that 
\begin{equation*}%\label{E:nabla-qd}
\langle \nabla H\rangle = \langle \partial f/\partial x_1, \dots, \partial f/\partial x_{a} \rangle 
\subset \langle \nabla f \rangle \ \text{and} \ \langle \nabla H\rangle \subset k[x_{a+1},\dots, x_{2a}].
\end{equation*}

\begin{comment}
Let $\rho$ be the one-parameter subgroup of $\SL(V)$ 
acting with weight $(n-a)/\gcd(n-a,a)$ on $\{x_i\}_{i=1}^a$
and weight $-a/\gcd(n-a,a)$ on $\{x_i\}_{i=a+1}^n$. Then $\langle\nabla f\rangle$
is invariant under $\rho$, and that $f$ is not $\rho$-invariant because
$G \neq 0$. 

Then 
\begin{equation}\label{E:nabla-qd}
\langle \nabla H\rangle \subset =\langle \partial f/\partial x_1, \dots, \partial f/\partial x_{\ell} 
\subset \langle \nabla f \rangle.
\end{equation}
is decomposable. 
\end{comment}

\end{proof}

\section{Direct sums and non-injectivity of $\nabla$} 
\label{S:injectivity}
Throughout this section, we work under:
\begin{assumption}\label{a-2} The ground field $k$ 
is algebraically closed and $\chark(k) \nmid (d+1)!$.
\end{assumption}
We explore the relationship between direct sum decomposability 
of a concise form $f\in S_{d+1}$ (that is, the form defining a hypersurface which is not a cone in 
$\PP^{n-1}$) and the non-injectivity of the gradient morphism (see Definition \ref{D:gradient})
\[
\nabla \colon \PP (S_{d+1})^{c} \to \Grass(n, S_d)
\] 
at the point
$\bar f \in \PP(S_{d+1})$. Our main result is the following complete characterization 
of concise forms that are uniquely determined by their gradient points, or, equivalently,
by their Jacobian ideals.
\begin{theorem}\label{T:injectivity-direct-sum} (A) Suppose $f\in S_{d+1}$ is a concise form. 
Then the following are equivalent:
\begin{enumerate}
\item $f$ is either a direct sum or an LDS form.
\item The morphism $\nabla$ is not injective at $\bar f$; that is, there exists $\bar g \neq \bar f \in \PP(S_{d+1})^{c}$ such that $\nabla f=\nabla g$.
\item The morphism $\nabla$ has positive fiber dimension at $\bar f$; that is, $\nabla^{-1}(\nabla \bar f)$ 
has dimension $\geq 1$ at $\bar f$.
\end{enumerate}
(B) In particular, if $f\in S_{d+1}$ is a GIT semistable form with respect to the standard $\SL(n)$-action,
then the following are equivalent:
\begin{enumerate}
\item $f$ is a direct sum.
\item The morphism $\nabla$ is not injective at $\bar f$.
%; that is, there exists $\bar g \neq \bar f \in \PP(S_{d+1})^{c}$ such that $\nabla f=\nabla g$.
\item The morphism $\nabla$ has positive fiber dimension at $\bar f$.%; 
%that is, $\nabla^{-1}(\nabla \bar f)$ has dimension $\geq 1$ at $\bar f$.
\end{enumerate}
\end{theorem}
Partial results along these lines were known earlier. In a 1983 paper \cite[\S4]{benson}, 
Max Benson proved this result on the locus
of smooth hypersurfaces. This was subsequently rediscovered several times,
for example in \cite[Lemma 3]{ueda} and \cite[Corollary 1.3]{wang}. 
In a more recent \cite[Theorem 1.1]{wang}, 
Wang proved that if $\nabla$ fails to be injective at $\bar f$,
then either $f$ is a direct sum or $f$ has a point of multiplicity $\deg(f)-1$. Note that all LDS forms
of degree $d+1$
have a point of multiplicity $d$ and are GIT non-semistable by Lemma \ref{L:lds}, but not all degree $d+1$ hypersurfaces with multiplicity $d$ points are GIT non-semistable. This shows that
Wang's result is slightly weaker
than our Theorem \ref{T:injectivity-direct-sum}. In fact, some applications (such as, for example,
in \cite{fedorchuk-isaev-smoothness}) require our stronger formulation.  

The key to our proof of Theorem \ref{T:injectivity-direct-sum} is Proposition \ref{P:benson},
which is based on an idea of Benson appearing in \cite[Proposition 4.1]{benson}
that was independently discovered
and communicated to me by Alexander Isaev. We note that Wang's result is proved by
the same method.  We obtain our sharp result by pushing Benson's method to its full logical
conclusion. 

\begin{proof}[{Proof of Theorem \ref{T:injectivity-direct-sum}}] 
We note that Part (B) follows immediately from Part (A) and Lemma \ref{L:lds}. We
proceed to prove Part (A).

Since $\nabla f=\nabla g$ implies that $\nabla (\lambda f+\mu g)=\nabla f$ for a general 
$[\lambda : \mu]\in \PP^1_{k}$, we have $(2)\Rightarrow (3)$, while the reverse implication is 
obvious.\footnote{Note however that over non-algebraically closed fields, 
$\nabla$ can be injective on $k$-points,
but still have positive fiber dimension at some of them. An example is given by $f=x^{d+1}+y^{d+1}\in  
\mathbb{F}_2[x,y]_{d+1}$.}

That $(1)\Rightarrow (2)$ when $f$ is a direct sum (respectively, a concise LDS form)
follows from Lemma \ref{L:gradient-ds}(2) (respectively, Lemma \ref{L:gradient-lds}(2)).

At last, we prove the implication $(2)\Rightarrow (1)$ in Proposition \ref{P:benson} below,
thus finishing the proof of the theorem.
\end{proof}

The following slightly more general result does not  
require the conciseness assumption.
\begin{prop}
\label{P:benson} Keep Assumption \ref{a-2}.
Suppose $f, g\in S_{d+1}$ are arbitrary nonzero forms that satisfy 
\begin{equation}
\left\langle \nabla g\right\rangle  \subset
\left\langle\nabla f\right\rangle,
\end{equation}
or, equivalently, $J_g \subset J_f$.
%\begin{equation}
%\left\langle \partial g/\partial x_1, \dots, \partial g/\partial x_n\right\rangle  \subset
%\left\langle \partial f/\partial x_1, \dots, \partial f/\partial x_n\right\rangle.  
%\end{equation}
Then either $g=\lambda f$, where $\lambda\in k$, or 
$f$ is LDS, or $f$ is a direct sum.
\end{prop}
\begin{remark} We regard any non-concise form as an instance of an LDS form 
obtained by setting $H=0$ in \eqref{E:lds}.
\end{remark}

\begin{proof} Following Benson \cite[\S4]{benson}, we note 
that there exists a basis $x_1,\dots, x_n$ of $V$ such that  
\begin{equation}\label{E:jnf}
\left(\begin{matrix} \dfrac{\partial g}{\partial x_1} \\ \vdots \\ \dfrac{\partial g}{\partial x_n}\end{matrix}\right)
=M\left(\begin{matrix} \dfrac{\partial f}{\partial x_1} \\ \vdots \\ \dfrac{\partial f}{\partial x_n}\end{matrix}\right),
\end{equation}
where $M$ is an $n\times n$ matrix in the Jordan normal form. If 
$M=\lambda \mathbb{I}_n$, where $\lambda\in k$, is a scalar multiple of the identity matrix, 
then by the Euler's formula and the assumption that $\chark(k)\nmid (d+1)!$, we have that $g=\lambda f$,
and we are done.

It remains to consider two cases: either $M$ has a Jordan block of size $>1$, or $M$ is diagonal
with at least two distinct eigenvalues. 

\begin{claim} If $M$ is not diagonal, then $f$ is LDS.
\end{claim}
\begin{proof}
Let $\lambda$ be any eigenvalue of $M$ with a Jordan block of size greater
than $1$. We assume that $J_1, \dots, J_\ell$ are the Jordan blocks with eigenvalue $\lambda$
of sizes $m_1,\dots, m_{\ell} \geq 2$ and that all other Jordan blocks $J_{\ell+1},\dots, J_{r}$,
of sizes $m_{\ell+1}, \dots, m_{r}$, respectively,
have either size $1$ or eigenvalue different from $\lambda$.

Double-index the variables so that 
\[
\{x_1,\dots,x_n\}=\{x_{j1},\dots, x_{jm_{j}}\}_{j=1}^{r}
\]
where $\{x_{j1},\dots, x_{jm_{j}}\}$ correspond to the Jordan block $J_{j}$ of size $m_j$. 

For each $j=1,\dots,\ell$, we suppose that  
\[
J_{j} = \left(\begin{matrix} \lambda & 0 & 0 &  \cdots \\ 
1 & \lambda & 0 & \cdots \\ \vdots & \ddots & \ddots & \vdots
 \\ \cdots & \cdots & 1 & \lambda  \end{matrix}\right).
\]
Using \eqref{E:jnf}, we then have
\begin{equation}\tag{$E_j$}
\begin{aligned}
\frac{\partial g}{\partial x_{j1}} &=\lambda \frac{\partial f}{\partial x_{j1}}, \\
\frac{\partial g}{\partial x_{j2}} &=\lambda \frac{\partial f}{\partial x_{j2}}+ \frac{\partial f}{\partial x_{j1}}, \\
\vdots & \\
\frac{\partial g}{\partial x_{jm_{j}}} &=\lambda \frac{\partial f}{\partial x_{jm_{j}}}+ \frac{\partial f}{\partial x_{jm_{j}-1}}.
\end{aligned}
\end{equation}
After passing to the second partials, and some elementary algebraic manipulations, we see that 
\begin{align}
%\frac{\partial^2 f}{\partial x_{j1}^2}=0, 
\frac{\partial^2 f}{\partial x_{j1}\partial x_{j't}}&=0, \ \text{for all 
$1\leq j, j' \leq \ell$, and $t=1,\dots, m_{j'}-1$}, \label{E:2nd-zero} \\
\frac{\partial^2 f}{\partial x_{j1}\partial x_{j't}}&=0, \ \text{for all $1\leq j\leq \ell<j'\leq r$ 
and all $t=1,\dots, m_{j'}$}, 
\label{E:2nd-zero-2} 
\\
\frac{\partial^2 f}{\partial x_{j1}\partial x_{j'm_{j'}}}&=\frac{\partial^2 f}{\partial x_{j1}\partial x_{j''m_{j''}}},
\ \text{for all $1\leq j, j', j'' \leq \ell$.} \label{E:mixed}
\end{align} 
Equation \eqref{E:2nd-zero} with $t=1$ implies that for all $1\leq j, j'\leq \ell$, we have
\[
\frac{\partial^2 f}{\partial x_{j1}\partial x_{j'1}}=0.
\]
It follows that $f$ is linear in the variables $x_{11}, \dots, x_{\ell 1}$, and so we can write
\[
f=x_{11}f_1+x_{21}f_2+\cdots+x_{\ell 1}f_{\ell}+G,
\]
where $f_{1},\dots,f_{\ell}, G$ do not involve $x_{11},\dots,x_{\ell 1}$. Moreover, we have that
\[
f_{j}=\frac{\partial f}{\partial x_{j1}},  \quad {j=1,\dots, \ell}.
\]
Equations \eqref{E:2nd-zero} and \eqref{E:2nd-zero-2} now imply that 
\[
f_1,\dots, f_{\ell} \in k[x_{jm_j}]_{j=1}^{\ell}.
\]
Next, \eqref{E:mixed} gives for all $1\leq j, j' \leq \ell$ that 
\[
\frac{\partial f_j}{\partial x_{j'm_{j'}}}=\frac{\partial f_{j'}}{\partial x_{j m_{j}}}.
\]
It follows that there exists $H\in k[x_{jm_j}]_{j=1}^{\ell}$ such that 
\[
f_j=\frac{\partial H}{\partial x_{jm_j}}.
\]
Re-indexing the variables so that $x_j=x_{j1}$ and $x_{\ell+j}=x_{jm_j}$ $j=1,\dots, \ell$, we
at last see that 
\[
f=\sum_{j=1}^{\ell} x_j \frac{\partial H(x_{\ell+1},\dots, x_{2\ell})}{\partial x_{j+\ell}}+G(x_{\ell+1},\dots,x_n)
\]
and so $f$ is LDS.
\end{proof}

\begin{claim} 
Suppose $M$ is diagonal. Then either $f$ is non-concise (and hence LDS) or $f$ is a direct sum.
\end{claim}
\begin{proof}
Let $\lambda$ be an eigenvalue of $M$, and $x_1,\dots, x_a$ are variables such that 
\[
\frac{\partial g}{\partial x_{i}} =\lambda \frac{\partial f}{\partial x_{i}}, \ i=1,\dots, a
\]
and
\[ 
\frac{\partial g}{\partial x_{j}} =\lambda_j \frac{\partial f}{\partial x_{j}}, \ j=a+1,\dots, n,
\]
where $\lambda_j\neq \lambda$ for $j>a$.  
Then $\dfrac{\partial^2 f}{\partial x_{i}x_{j}}=0$ for all $i\leq a$ and $j\geq a+1$. Hence
 $f=f_1(x_1,\dots,x_a)+f_2(x_{a+1},\dots,x_n)$.
 %either $f$ is degenerate or $f$ is a direct sum.
 \end{proof}

This finishes the proof of proposition.
\end{proof}

We obtain a number of consequences of Theorem \ref{T:injectivity-direct-sum} and 
Proposition \ref{P:benson}, that in particular describe the fibers of the gradient morphism.
To state these results, we need one more piece of terminology.
Recall from \cite[\S 3.1]{apolarity} that for a non-degenerate form $f\in S_{d+1}$, a decomposition 
\[
f=f_1+\cdots+f_r,
\]
is called a \emph{maximally fine direct sum decomposition} if 
$V=E(f_1)\oplus \cdots \oplus E(f_r)$, and $f_i$ 
is not a direct sum in $\Sym^{d+1} E(f_i)$,
for all $i=1,\dots, r$.

\begin{corollary}[Fibers of $\nabla$]\label{C:fiber-nabla}
Suppose $f\in S_{d+1}$ is not an LDS form and $f=f_1+\cdots+f_r$ 
is a maximally fine direct sum decomposition. Suppose $g\in S_{d+1}$ is such 
that $\langle\nabla g\rangle\subset \langle\nabla f\rangle$. Then $g\in \langle f_1,\dots, f_r\rangle$. In particular,
for the gradient morphism $\nabla \colon \PP(S_{d+1})^{c} \to \Grass(n, S_d)$, we have
\begin{equation}
\label{E:fiber-nabla}
\nabla^{-1}(\nabla f)=\{\lambda_1 f_1+\lambda_2 f_2+\cdots+\lambda_r f_r\mid \lambda_i\in k^*\}.
\end{equation}
\end{corollary}
\begin{proof}
Suppose $\nabla g\subset \nabla f$. Let $x_1,\dots, x_n$ be a basis of $V$ adapted to the direct sum
decomposition $E(f)=E(f_1)\oplus\cdots\oplus E(f_r)$.
Take $x_i \in E(f_a)$ and $x_j \in E(f_b)$, where $a\neq b$. Then we 
have that 
\[
\dfrac{\partial^2 g}{\partial x_i\partial x_j} \in \Sym^{d} E(f_a) \cap \Sym^{d} E(f_b) =(0),
\] 
(cf. the proof of Lemma \ref{L:gradient}). 
It follows that the partials $\dfrac{\partial^2 g}{\partial x_i\partial x_j}$ vanish whenever 
 $x_i \in E(f_a)$ and $x_j \in E(f_b)$, where $a\neq b$. This implies that we can write
\[
g=g_1+\cdots+g_r,
\] where $E(g_i)=E(f_i)$. Moreover, from $\langle\nabla g\rangle\subset \langle\nabla f\rangle$, 
it follows that 
$\langle\nabla g_i\rangle \subset \langle\nabla f_i\rangle$. Applying Proposition \ref{P:benson}, and the fact that 
each $f_i$ is not an LDS form, we conclude 
that $g_i$ is a scalar multiple of $f_i$. The claim follows.
\end{proof}

\begin{corollary} Let $\nabla \colon \PP(S_{d+1})^{c} \to \Grass(n, S_d)$ be the gradient 
morphism. Then the non-injectivity locus of $\nabla$ is equal to the union of the direct sum
locus and the locus of LDS forms:
\begin{multline*}
\left\{ \bar f\in \PP(S_{d+1})^{c} \mid \text{there exists $\bar g\neq \bar f \in \PP(S_{d+1})^{nd}$ such that
 $\nabla f=\nabla g$}
\right\} \\
= \left\{ \bar f\in \PP(S_{d+1})^{c} \mid \text{either $f$ is a direct sum or $f$ is an
LDS form}
\right\}.
\end{multline*}

\end{corollary}

For concise forms of degree $d+1\geq 3$, Kleppe has established that a maximally fine direct sum 
decomposition is unique \cite[Theorem 3.7]{kleppe-thesis}.  We obtain the following 
generalization of his result:
\begin{corollary}\label{C:mfds-unique}
Suppose $f$ is not an LDS form. Then
a maximally fine direct sum decomposition of $f$ is unique.
\end{corollary} 
\begin{proof}
Suppose $f=f_1+\cdots+f_s=g_1+\cdots+g_t$ are two maximally fine direct sum decompositions.
Then by Corollary \ref{C:fiber-nabla}, we have that 
\[
\{\lambda_1 f_1+\lambda_2 f_2+\cdots+\lambda_r f_r\mid \lambda_i\in k^*\}=
\nabla^{-1}(\nabla f)=\{\lambda_1 g_1+\lambda_2 g_2+\cdots+\lambda_t g_t\mid \lambda_i\in k^*\}
.\]
It follows that
$g_i \in \langle f_1,\dots, f_r\rangle$.  Since $g_i$ is not a
direct sum itself, we must have $g_i=c_i f_j$ for some $j=1,\dots,s$ and some $c_i\in k^*$.  It is then 
immediate that $t=s$ and $c_i=1$ for all $i=1,\dots, s$. 
%claim follows. 
\end{proof}

\subsubsection{Connection with the work of Buczy\'nska-Buczy\'nski-Kleppe-Teitler}
\label{S:prior-work} 

\begin{comment}
In \cite{apolarity}, Buczy\'nska, Buczy\'nski, Kleppe, and Teitler prove
that for a non-degenerate form $f\in S_{d+1}$ over an algebraically closed field,
the apolar ideal $f^{\perp}$ has 
a minimal generator in degree $d+1$ if and only if either $f$ is a direct sum, or $f$ is 
an LDS form. 
Since LDS forms are GIT non-semistable and in particular 
singular, this translates into a computable and effective criterion for \emph{recognizing} whether a smooth form 
$f$ is a direct sum over an algebraically closed field.

In his thesis \cite{kleppe-thesis}, Kleppe uses the quadratic part of the apolar ideal $f^{\perp}$ to 
define an associative algebra $M(f)$ of finite dimension over the base field
($M(f)$ is different from the Milnor algebra $M_f$). He
then proves that, over an arbitrary field, direct sum decompositions of $f$ are 
in bijection with complete sets of orthogonal idempotents
of $M(f)$. 
A key step in the proof of the direct sum criterion in \cite{apolarity} is the
Jordan normal form decomposition of a certain linear operator, which in general requires solving a
characteristic equation. 
Similarly, finding a complete set of orthogonal idempotents requires solving a system of quadratic equations. This makes it challenging to turn \cite{apolarity} or \cite{kleppe-thesis} into an algorithm for finding
direct sum decompositions when they exist. 
\end{comment}

The following result is established in \cite{apolarity} under the assumption that $k$ is an algebraically
closed field of characteristic $0$: 
\begin{theorem}[{see \cite[Theorem 1.7]{apolarity}}]
A concise form $f\in S_{d+1}$ is either a direct sum or an LDS form if and only if the apolar ideal $f^{\perp}$
has a minimal generator in degree $d+1$. 
\end{theorem}
Since LDS forms are GIT non-semistable and in particular 
singular, this translates into a computable and effective criterion for \emph{recognizing} whether a smooth form 
$f$ is a direct sum over an algebraically closed field.

The following simple observation 
reconciles the above result with Theorem \ref{T:injectivity-direct-sum}. 
\begin{prop} Keep Assumption \ref{a-2}.  Consider the apolarity action $\D \times S \to S$ given by \[
F(z_1,\dots,z_n) \circ g=F\left(\frac{\partial}{\partial x_1}, \dots, \frac{\partial}{\partial x_n}\right) g(x_1,\dots,x_n).
\]
Then for a form $f\in S_{d+1}$, the apolar ideal 
\[
f^{\perp} =\{F\in \D \mid F\circ f=0\}
\]
has a minimal generator in degree $d+1$ if and only if 
there exists $g\in S_{d+1}$, not a scalar multiple of $f$, such that $\nabla f \subset \nabla g$.
\end{prop}
\begin{proof}
Let $I=f^{\perp} \subset \D$. Then  $S/I$ is a Gorenstein Artin $k$-algebra with socle in degree $d+1$.
%we have $\codim (I_{d+1}, S_{d+1})=1$.
%Let $J=(I_n \mid n\leq d)$ be the ideal generated by the elements
%of $I$ of degree at most $d$.    We have $J_{d+1} \subset I_{d+1}$ 
%We let $J_{d+1}=\D_{1} I_d$. Then $J_{d+1} \subset I_{d+1}$ 
By definition, $I$ has a minimal generator in degree $d+1$ if and only if 
and $\dim_k (I_{d+1}/ \D_{1} I_d) \geq 1$, or equivalently if $\D_1 I_d \subsetneq I_{d+1}$.

On the other hand, $\D_1 I_d \subsetneq I_{d+1}$ if and only if 
\begin{equation*}
\{g \in S_{d+1} \mid F\circ g =0,\  \text{for all $F \in I_{d+1}$}\}
 \subsetneq 
\{g\in S_{d+1} \mid F\circ g=0\ \text{for all $F\in \D_1 I_{d}$}\}.
\end{equation*}
We now compute that
\[
\{g \in S_{d+1} \mid F\circ g =0,\  \text{for all $F \in I_{d+1}$}\} =\langle f\rangle,
\]
and
\begin{multline*}
\{g\in S_{d+1} \mid F\circ g=0\ \text{for all $F\in \D_1 I_{d}$}\} \\ =  
\{g\in S_{d+1} \mid F\circ \frac{\partial g}{\partial
x_i}=0\ \text{for all $F\in I_{d}$, and all $i=1,\dots, n$}\}\\ 
=\{g\in S_{d+1} \mid F\circ \frac{\partial g}{\partial
x_i}=0\ \text{for all $F\in (f^{\perp})_{d}$, and all $i=1,\dots, n$}\}
%&=\{g\in S_{d+1} \mid F\circ \frac{\partial g}{\partial
%x_i}=0\ \text{for all $F\in (f^{\perp})_{d}$}\} 
%\\ 
=\{g\in S_{d+1} \mid \nabla f\subset \nabla g\},
\end{multline*}
where we have used the equality 
\[
I_d=(f^{\perp})_d=\{F\in \D_d \mid F\circ \partial f/\partial x_i=0 \quad \text{for all $i=1,\dots, n$}\}.
\]
We conclude that $\D_1 I_d \subsetneq I_{d+1}$ if and only if 
\[
\langle f \rangle \subsetneq \{g\in S_{d+1} \mid \nabla f\subset \nabla g\},
\]
which is precisely the condition that 
there exists $g\in S_{d+1}$, not a scalar multiple of $f$, such that $\nabla f \subset \nabla g$.

\begin{comment}
At the same time considering the apolarity action $\D \times S \to S$ given by \[
F(z_1,\dots,z_n) \circ g=F\left(\frac{\partial}{\partial x_1}, \dots, \frac{\partial}{\partial x_n}\right) g(x_1,\dots,x_n),
\]
we have using $J_{d+1}=(z_1,\dots,z_n)J_d$:
\begin{multline*}
\{g\in S_{d+1} \mid F\circ g=0\ \text{for all $F\in J_{d+1}$}\} =\{g\in S_{d+1} \mid F\circ \frac{\partial g}{\partial
x_i}=0\ \text{for all $F\in J_{d}$}\}\\ =\{g\in S_{d+1} \mid F\circ \frac{\partial g}{\partial
x_i}=0\ \text{for all $F\in (f^{\perp})_{d}$}\}=\{g\in S_{d+1} \mid F\circ \frac{\partial g}{\partial
x_i}=0\ \text{for all $F\in (f^{\perp})_{d}$}\}\\=\{g\in S_{d+1} \mid \nabla f\subset \nabla g\}.
\end{multline*}
\end{comment}

\end{proof}

\section{Direct sum decomposability of smooth forms}
\label{S:factorization}
We keep Assumption \ref{ass-nd}. 
\begin{theorem}\label{MT1} Suppose $k$ is an arbitrary field satisfying $\chark(k)\nmid (n(d-1))!$. Let $f\in S_{d+1}$ be a smooth form.
Then the following are equivalent:
\begin{enumerate}
\item\label{direct} $f$ is a direct sum.
\item\label{nabla-direct} $\nabla f$ is a balanced direct sum. 
\item\label{A-direct} $A(f)$ is a balanced direct product.
\item\label{A-direct-nonbalanced} $A(f)$ is a direct product.
\item\label{nabla-action-1} $\nabla f$ admits a non-trivial $\GG_m$-action defined over $k$.
\item\label{A-action} $A(f)$ admits a non-trivial $\GG_m$-action defined over $k$.
\end{enumerate}
Moreover, if $z_1,\dots, z_n$ is a basis of $V^{\vee}$ in which $A(f)$ factors as 
\[
A(f)=G_1(z_1,\dots,z_a)G_2(z_{a+1},\dots, z_n),
\] 
then $f$ decomposes as 
\[f=f_1(x_1,\dots,x_a)
+f_2(x_{a+1},\dots,x_n)\] in the dual basis $x_1,\dots,x_n$ of $V$.
\end{theorem}

\begin{proof}[Proof of Theorem \ref{MT1}] 

The implications $\eqref{direct} \Longrightarrow \eqref{nabla-direct}$ and
$\eqref{direct}  \Longrightarrow \eqref{nabla-action-1}$ 
are in Lemma \ref{L:gradient-ds}. 
The implication $\eqref{nabla-direct} \Longrightarrow \eqref{direct}$ is in Lemma \ref{L:gradient}.

The equivalence $\eqref{nabla-direct} \Longleftrightarrow \eqref{A-direct}$ is proved in Proposition \ref{P:key} below. 
This concludes the proof of equivalence for the first three conditions. 

Next we prove $\eqref{A-direct-nonbalanced} \Longrightarrow \eqref{A-direct}$.
Suppose $A(f)=G_1(z_1,\dots,z_a)G_2(z_{a+1},\dots,z_n)$ 
is a direct product decomposition in 
a basis $z_1,\dots,z_n$ of $V^{\vee}$. Let $x_1,\dots,x_n$ be the dual basis of $V$. Suppose $x_1^{d_1}\cdots x_n^{d_n}$ is 
the smallest with respect to the graded reverse lexicographic order monomial of degree $n(d-1)$ that does not lie in $(J_f)_{n(d-1)}$. Since
$z_1^{d_1}\cdots z_n^{d_n}$ must appear with a nonzero coefficient in $A(f)$, we have that  
\[
d_1+\cdots+d_a=\deg G_1.
\] On the other hand,
by \cite[Lemma 4.1]{fedorchuk-isaev}, we have that $d_1+\cdots+d_a\leq a(d-1)$.
It follows that $\deg G_1\leq a(d-1)$. By symmetry, we also have that $\deg G_2 \leq (n-a)(d-1)$.
We conclude that both inequalities must be equalities and so $A(f)=G_1G_2$ 
is a balanced direct product decomposition. Alternatively, we can consider a diagonal action of $\GG_m
\subset \SL(V)$ on $V$
 that acts on $V^{\vee}$ as follows: 
\[
t\cdot (z_1,\dots, z_n)=\left(t^{(n-a)} z_1,\dots, t^{(n-a)} z_a, t^{-a} z_{a+1},
\dots, t^{-a} z_{n}\right).
\]
Then $A(f)$ is homogeneous with respect to this action, and has weight
$(n-a)\deg G_1 -a \deg G_2.$
However, the relevant parts of the proof of \cite[Theorem 1.2]{fedorchuk-ss} go through to show that $A(f)$ satisfies 
the Hilbert-Mumford numerical criterion for semistability. This forces $(n-a)\deg G_1 -a \deg G_2=0.$

We now turn to the last two conditions. First, the morphism $\AA$ is an  $\SL(n)$-equivariant  locally closed immersion by 
\cite[\S2.5]{alper-isaev-assoc-binary}, and so is stabilizer preserving. 
This proves the equivalence $\eqref{nabla-action-1} \Longleftrightarrow \eqref{A-action}$. 
The implication $\eqref{nabla-action-1}  \Longrightarrow \eqref{direct}$ follows
from the proof of \cite[Theorem 1.0.1]{fedorchuk-ss} that shows that for a smooth $f$, 
the gradient point $\nabla f$ has a non-trivial $\GG_m$-action if and only if $f$ is a direct sum. We note that even though 
stated over $\CC$, the relevant parts of the proof of \cite[Theorem 1.0.1]{fedorchuk-ss}  use only \cite[Lemma 3.5]{fedorchuk-ss}, which remains valid over a field $k$ with $\chark(k)=0$ or $\chark(k)>d+1$, and the 
fact that a smooth form over any field must satisfy the Hilbert-Mumford numerical criterion for stability.
\end{proof}

\begin{comment}
\begin{remark} Given a basis of $V$ diagonalizing a non-trivial $\GG_m$-action on $\nabla F$,
the proof of \cite[Theorem 1.0.1]{fedorchuk-ss} can be turned into an algorithm for finding a basis
of $V$ in which $F$ decomposes into a direct sum.  At the same time, using a tangent space computation,
it is possible to give a purely linear-algebraic criterion for $\nabla F$ to admit a non-trivial $\GG_m$ action,
without writing one down explicitly. 
\end{remark}
\end{comment}

\begin{comment}
\begin{proof}[Proof of Theorem \ref{MT2}] 
By \cite[Theorem 1.0.1]{fedorchuk-ss}, for every GIT stable $f$, 
the gradient point $\nabla f$ is polystable. Furthermore, it admits a $\GG_m$-action if and 
only if $f$ is a direct sum. 
Moreover, \cite[Proposition 2.1]{fedorchuk-ss} shows that the morphism of the GIT quotients
\[
\nabla\gitq \SL(n)\colon (\PP S_{d+1})^{s}\gitq \SL(n) \to \Grass(n, S_d)^{ss}\gitq \SL(n)
\] is injective. This proves 
that for every stable $f$, the fiber dimension of $\nabla$ at 
$\langle f\rangle$ equals to the dimension of the stabilizer of $\nabla f$. 
This concludes the proof of all equivalences. The fact that the locus of direct 
sums is closed in $(\PP S_{d+1})^{s}$ now follows from the upper semicontinuity 
(on the domain) of fiber dimensions.
\end{proof}
\end{comment}

\begin{prop}\label{P:key} Let $d\geq 2$. Suppose $k$ is a field with $\chark(k)=0$ or $\chark(k)>n(d-1)$. Then an element
$U\in \Grass(n, \Sym^{d} V)_{\Res}$ is a balanced direct sum if and only if $\AA(U)$ is a balanced
direct product. Moreover, if $z_1,\dots, z_n$ is a basis of $V^{\vee}$ in which $\AA(U)$ factors as a balanced direct product,
%$\AA(U)=G_1(z_1,\dots,z_a)G_2(z_{a+1},\dots, z_n)$,
 then $U$ decomposes as a balanced direct sum in the dual basis $x_1,\dots,x_n$ of $V$.
\end{prop} 

\begin{proof}
The forward implication is an easy observation. 
Consider a balanced direct sum $U=\langle g_1,\dots, g_n\rangle
\in \Grass(n, k[x_1,\dots,x_n]_{d})_{\Res},$
where $g_1,\dots, g_a\in k[x_1,\dots,x_a]_{d}$ and $g_{a+1},\dots, g_n\in k[x_{a+1},\dots,x_n]_{d}$. Then, up to a nonzero scalar, 
\[
\AA(U)=\AA(g_1,\dots,g_a)\AA(g_{a+1},\dots,g_n),
\]
where $\AA(g_1,\dots,g_a)\in k[z_1,\dots,z_a]_{a(d-1)}$ and 
$\AA(g_{a+1},\dots,g_n)\in k[z_{a+1},\dots, z_n]_{(n-a)(d-1)}$;
see \cite[Lemma 2.11]{fedorchuk-isaev}, which also follows from the fact that
on the level of algebras, we have
\[
\frac{k[x_1,\dots,x_n]}{(g_1,\dots,g_n)} \simeq \frac{k[x_1,\dots,x_a]}{(g_1,\dots, g_a)} \otimes_{k} \frac{k[x_{a+1},\dots,x_n]}{(g_{a+1},\dots, g_n)} .
\]

Suppose now $\AA(U)$ is a balanced direct product in a basis $z_1,\dots,z_n$ of $V^{\vee}$:
\begin{equation}\label{product}
\AA(U)=F_1(z_1,\dots,z_a)F_2(z_{a+1},\dots,z_n),
\end{equation}
where $\deg(F_1)=a(d-1)$ and $\deg(F_2)=(n-a)(d-1)$.
Let $x_1,\dots,x_n$ be the dual basis of $V$, and let $I_U \subset k[x_1,\dots,x_n]$ be the 
complete intersection ideal spanned by the elements of $U$.
We have that  \[I_U=\AA(U)^{\perp} \subset k[x_1,\dots,x_n].\] %by the definition of $\AA(U)$.
It is then evident from \eqref{product} and the definition of an apolar ideal that 
\begin{equation}\label{EE1}
(x_1,\dots, x_a)^{a(d-1)+1} \subset I_U
\end{equation}
and 
\begin{equation}\label{EE2}
(x_{a+1}, \dots,x_n)^{(n-a)(d-1)+1} \subset I_U.
\end{equation}

We also have the following observation:
\begin{claim}\label{claim} %We have that 
\begin{align*} 
\dim_k \bigl(U \cap (x_1,\dots, x_a)\bigr) &= a, \\ 
\dim_k \bigl(U \cap (x_{a+1},\dots, x_n)\bigr) &= n-a.
\end{align*}
\end{claim}
\begin{proof}
By symmetry, it suffices to prove the second statement. Since $U$ is spanned by a length $n$ regular sequence of
degree $d$ forms,
we have that $\dim_k \bigl(U \cap (x_{a+1},\dots, x_n)\bigr) \leq n-a$. 
Suppose we have a strict inequality.
Let \[
R:=k[x_1,\dots, x_n]/(I_U, x_{a+1},\dots, x_n) \simeq k[x_1,\dots, x_a]/I'.
\] Then $I'$ is generated in degree $d$, and has at least $a+1$ minimal generators in that degree.
It follows that the top degree of $R$ is strictly less than $a(d-1)$, and so
$I'_{a(d-1)}=k[x_1,\dots, x_a]_{a(d-1)}$ (cf. \cite[Lemma 2.7]{fedorchuk-isaev}). But then 
\[k[x_1,\dots, x_a]_{a(d-1)} \subset (x_{a+1},\dots, x_n) + I_U.\] 
Using \eqref{EE2}, this gives \[
k[x_1,\dots, x_a]_{a(d-1)}  k[x_{a+1},\dots, x_n]_{(n-a)(d-1)} \subset I_U.\] 
Thus every monomial of \[k[z_1,\dots, z_a]_{a(d-1)}  k[z_{a+1},\dots, z_n]_{(n-a)(d-1)}\]
appears with coefficient $0$ in $\AA(U)$, which contradicts \eqref{product}.
\end{proof}

At this point, we can apply \cite[Proposition 3.1]{fedorchuk-isaev} to conclude 
that $U \cap k[x_1,\dots, x_a]_{d}$ contains a regular sequence of length $a$ and
that $U \cap k[x_{a+1},\dots, x_n]_{d}$ contains a regular sequence of length $n-a$. This shows
that $U$ decomposes as a balanced direct sum in the basis $x_1,\dots,x_n$ of $V$. 
However, for the sake of self-containedness, we proceed to give a more direct argument:
%not invoking the full strength of \cite[Prop. 3.1]{fedorchuk-isaev}.

By Claim \eqref{claim}, there exists a regular sequence $s_1,\dots, s_a \in k[x_1,\dots, x_a]_{d}$ such that 
\[
(s_1,\dots, s_a)=\bigl(g_1(x_1,\dots, x_a, 0, \dots, 0),  \dots, g_n(x_1,\dots, x_a, 0, \dots, 0)\bigr)
\]
and a regular sequence $t_1,\dots, t_{n-a} \in k[x_{a+1},\dots, x_n]_{d}$ such that 
\[
(t_1,\dots, t_{n-a})=\bigl(g_1(0,\dots, 0, x_{a+1}, \dots, x_n),  \dots, g_n(0,\dots, 0, x_{a+1}, \dots, x_n)\bigr).
\]
Let \[
W:=\langle s_1,\dots, s_a, t_1,\dots, t_{n-a}\rangle \in \Grass(n, S_d)_{\Res}
\] and let $I_W$ be the ideal generated by $W$.
We are going to prove that $U=W$, which will conclude the proof of the proposition.

Since $\chark(k)=0$ or $\chark(k)>n(d-1)$, Macaulay's theorem applies, and so 
to prove that $U=W$, we need to show that the ideals $I_U$ and $I_W$
coincide in degree $n(d-1)$. For this, it suffices to prove that $(I_W)_{n(d-1)}\subset (I_U)_{n(d-1)}$. 

Since $s_1,\dots, s_a$ is a regular sequence in $k[x_1,\dots,x_a]_{d}$, we have that \[k[x_1,\dots, x_a]_{a(d-1)+1}\subset (s_1,\dots, s_a).\] Similarly, we have that 
\[k[x_{a+1},\dots, x_n]_{(n-a)(d-1)+1}\subset (t_1,\dots, t_{n-a}).\] Together with
\eqref{EE1} and \eqref{EE2}, this gives
\begin{equation}\label{E:J}
(x_1,\dots, x_a)^{a(d-1)+1} + (x_{a+1}, \dots,x_n)^{(n-a)(d-1)+1} \subset I_U \cap I_W.
\end{equation}
Set $J:=(x_1,\dots, x_a)^{a(d-1)+1} + (x_{a+1}, \dots,x_n)^{(n-a)(d-1)+1}$. It remains
to show that \[(I_W)_{n(d-1)}\subset (I_U)_{n(d-1)} + J_{n(d-1)}.\] 
To this end, consider
\[
\sum_{i=1}^{a} q_i s_i +\sum_{j=1}^{n-a} r_j t_j \in (I_W)_{n(d-1)},
\]
where $q_1,\dots, q_a, r_1,\dots, r_{n-a} \in S_{n(d-1)-d}$.
Since $s_1,\dots, s_a \in k[x_1,\dots, x_a]_d$, and we are working modulo $J$, we can assume that $q_i \in (x_{a+1}, \dots, x_n)^{(n-a)(d-1)}$,
for all $i=1,\dots, a$. Similarly, we can assume that $r_j \in (x_1,\dots, x_a)^{a(d-1)}$, for all $j=1,\dots,n-a$. 

By construction, we have $s_1, \dots, s_a \in I_U + (x_{a+1}, \dots, x_n)$ and $t_1,\dots, t_{n-a} \in I_U + (x_1,\dots, x_a)$. Using this, and \eqref{E:J}, we
conclude that 
\[
\sum_{i=1}^{a} q_i s_i +\sum_{j=1}^{n-a} r_j t_j \in I_U + J.
\]
This finishes the proof of the proposition.
\end{proof}

We use  
Theorem \ref{MT1} to give an alternate proof of Corollary \ref{C:mfds-unique} for smooth forms, 
deducing it from the fact that a polynomial
ring over a field is a UFD:
\begin{prop}\label{P:uniqueness} %Let $d\geq 2$. %Suppose $\chark(k)\nmid (n(d-1))!$.  
Keep Assumption \ref{ass-nd}.
Suppose $f\in S_{d+1}$ is a smooth form. Then $f$ has a unique maximally fine direct sum decomposition. 
\end{prop}
\begin{proof}[Proof of Proposition \ref{P:uniqueness}] 
%If $\chark(k)=3$,
%the case of $(n,d)=(2,2)$ is vacuous since no smooth binary cubic will be a direct sum. In all other cases,
%$\chark(k)>n(d-1)$ implies $\chark(k)>\max\{n(d-1),d+1\}$.

Suppose 
$f=f_1+\cdots+f_s=g_1+\cdots+g_t$ are two maximally fine direct sum decompositions.
Then
\[
A(f_1)\cdots A(f_s)=A(g_1)\cdots A(g_t) \in \PP \D_{n(d-1)}, %\Sym^{n(d-1)} V^{\vee},
\]
where $V^{\vee}=\oplus_{i=1}^s E(A(f_{i}))=\oplus_{j=1}^t E(A(g_{j}))$.
Suppose some $A(f_i)$ shares irreducible factors with more than one $A(g_j)$.
Then by the uniqueness of factorization in $\D$, we must have a non-trivial factorization
$A(f_i)=G_1G_2$ such that $E(G_1)\cap E(G_2)=(0)$. Then $A(f_i)$ is a direct
product, and so $f_i$ must be a direct sum by Theorem \ref{MT1}, contradicting the maximality assumption. 
Therefore, no $A(f_i)$ shares an irreducible factor with more than one $A(g_j)$; and, by symmetry, 
no $A(g_j)$ shares an irreducible factor with more than one $A(f_i)$. It follows that $s=t$ and, up to reordering,
$A(f_i)=A(g_i)$, and thus $E(A(f_i))=E(A(g_i))$, for all $i=1,\dots, t$. We conclude that $E(f_i)=E(g_i)$, which using $f_1+\cdots+f_t=g_1+\cdots+g_t$
forces $f_i=g_i$, for all $i=1,\dots, t$. 
\end{proof}

\section{Necessary conditions for direct sum decomposability}
\label{S:necessary}

Our next two results give
easily verifiable necessary conditions for an arbitrary form to be a direct sum. They hold over 
a large enough field, with no restriction on characteristic, and are 
independent of the results of Sections \ref{S:injectivity} and \ref{S:factorization}.
We apply them to prove that determinant and pfaffian-like polynomials are not direct sums
in Corollaries \ref{C:determinant} and \ref{C:pfaffian}.
 We keep notation of \S\ref{S:notation}.

\begin{theorem}\label{MT3} Suppose $f$ is a form in $S=k[x_1,\dots,x_n]_{d}$
and that $\card(k) \gg d$. 
\begin{enumerate}
\item Let $b=\dim_k \nabla f$. If $f$ has a factor $g$ such that 
$\dim_k \nabla g\leq \lfloor \frac{b-1}{2}\rfloor$, then $f$ is not a direct sum.
\item If $f$ has a repeated factor, then $f$ is not a direct sum.
\end{enumerate}
\end{theorem}
\begin{remark} It is possible to pinpoint precisely how large the cardinality of $k$
has to be, but since we do not have an application for this, we will not do so. 
\end{remark}

\begin{corollary} 
Suppose $f$ is a form with $\dim_k \nabla f \geq 3$, and that 
$\card(k) \gg \deg(f)$. 
If $f$ has a linear factor, then $f$ is not a direct sum.
\end{corollary}

The result of this corollary was proved in \cite[Proposition 2.12]{apolarity} using 
a criterion of Smith and Stong \cite{smith-stong} for indecomposability of  Gorenstein Artin algebras into connected sums.
Our proof of the linear factor case of Theorem \ref{MT3} and the statement for higher degree factors appear to be new.

\begin{proof}[Proof of Theorem \ref{MT3}]
We apply Lemma \ref{L:gradient-ds}. For (1), suppose $f=gh$, %where $g$ is a factor such that
%$\dim_k \nabla g\leq \lfloor \frac{b-1}{2}\rfloor$ 
and that in some basis of $V$ we have \[
%gh=
f=f_1(x_1,\dots,x_a)+f_2(x_{a+1},\dots,x_n)\in S,
\] 
and that $\dim_k \nabla f=b$ while $\dim_k \nabla g \leq \lfloor \frac{b-1}{2}\rfloor$.
Let $\rho$ be the 1-PS subgroup of $\SL(V)$ acting with weight $w_1:=\frac{n-a}{\gcd(n-a,a)}$ on $\{x_i\}_{i=1}^a$
and weight $w_2:=-\frac{a}{\gcd(n-a,a)}$ on $\{x_i\}_{i=a+1}^n$. 
Then $\rho$ fixes $\nabla f$, and 
\[
\langle \nabla f\rangle=\langle \nabla f\rangle_{(w_1d)}\oplus \langle \nabla f\rangle_{(w_2 d)}
\] is the decomposition into the $\rho$-weight-spaces. By the assumption on the cardinality of $k$, these two weight subspaces are distinct.  

Since $\langle\nabla f\rangle \subset g\langle\nabla h\rangle+h \langle\nabla g\rangle$, we have
\[
\dim_k(\langle\nabla f\rangle \cap g\langle\nabla h\rangle) \geq b-\dim_k h\langle\nabla g\rangle=b-\dim_k \langle\nabla g\rangle  \geq \left\lceil \frac{b+1}{2}\right\rceil.
\]
It follows by dimension considerations that 
some nonzero multiple $gr$, where $r\in \langle\nabla h\rangle$, 
belongs to one of the two weight-spaces $\langle \nabla f\rangle_{(w_id)}$ of $\rho$ in $\nabla f$.  
Thus $g$ itself is homogeneous with respect to $\rho$. Again invoking the assumption
on the cardinality of $k$, we conclude that 
either $g\in k[x_1,\dots,x_a]$ or $g\in k[x_{a+1},\dots, x_n]$. This forces
either $f_2=0$ or $f_1=0$, respectively. A contradiction!

For (2), suppose $f$  is a direct sum with a repeated factor $g$. Let $\rho$ be 
the 1-PS of $\SL(V)$ as above. Since $\nabla f \subset (g)$, some nonzero multiple $gr$, where
$r\in S_{d-\deg g}$, belongs to one of the two weight-spaces 
of $\rho$ in $\nabla f$. It follows that $g$ is homogeneous with respect to $\rho$
and so we obtain a contradiction as in (1).
\end{proof}

Our next result needs the following definition that is standard in the theory of GIT stability of
homogeneous forms:

\begin{definition}\label{D:state}
Given a basis $x_1,\dots, x_n$ of $V$ and a nonzero $f\in S_{d}$, we define \emph{the state of $f$} to be
the set of multi-indices $\Xi(f) \subset \{(d_1,\dots,d_n)\in \ZZ_{\geq 0}^n \mid d_1+\cdots+d_n=d\}$ such that 
\[
f= \sum_{(d_1,\dots,d_n) \in \Xi(f)} a_{(d_1,\dots,d_n)}x_1^{d_1}\cdots x_n^{d_n}, \quad \text{where $a_{(d_1,\dots,d_n)}\in k^*$.}
\]
In other words, the state of $f$ is the set of monomials appearing with nonzero coefficient in $f$.
We set $\Xi(0)=\varnothing$.
\end{definition}

\begin{theorem}\label{MT4} Suppose $k\neq \mathbb{F}_2$.
Suppose $f\in S_{d}$, where $d\geq 3$,  is  
such that 
in some basis $x_1,\dots, x_n$ of $V$ the following conditions hold:
\begin{enumerate}
\item \label{E:partial-nonzero}
$\Xi(\partial f/\partial x_i)\neq \varnothing$ for all $1\leq i\leq n$.
\item\label{E:partial-states-non-overlap}
%\begin{equation}\label{E:partial-states-non-overlap}
$\Xi(\partial f/\partial x_i) \cap \Xi(\partial f/\partial x_j) = \varnothing$ for all $1\leq i<j\leq n$.
%\end{equation}
\item \label{E:recoverable-state} 
\begin{multline*} 
\Xi(f)= \Big\{(d_1,\dots,d_n)%\in \ZZ_{\geq 0}^n \ 
\Big\vert \text{ $\chark(k)\nmid d_i$ for some $1\leq i \leq n$ and } \\ \Xi\left(\frac{\partial (x_1^{d_1}\cdots x_n^{d_n})}{\partial x_i}\right) \subset \cup_{j=1}^n \Xi(\partial f/\partial x_j), \ \text{for all $1\leq i \leq n$}\Big\}.\end{multline*}
\item\label{E:connected} The graph with the vertices in $\{1,\dots, n\}$ and the edges given by 
\[
\left\{(ij) \mid \frac{\partial^2 f}{\partial x_i \partial x_j} \neq 0\right\}
\]
is connected. 
\end{enumerate}
Then $f$ is not a direct sum.
\end{theorem}
\begin{remark} In words, \eqref{E:partial-states-non-overlap} says that no two first partials 
of $f$ share a common
monomial, and \eqref{E:recoverable-state} says that any monomial all of whose nonzero first 
partials appear in first partials of $f$ must appear in $f$. 
%In characteristic 0, the last condition 
%of course says that $f$ is not a direct sum in the basis $x_1,\dots,x_n$.
\end{remark}

As an immediate corollary of this theorem, we show that the  
$n\times n$ generic determinant and permanent polynomials, and the $2n\times 2n$ generic pfaffian polynomials, as well as
any other polynomial of the same state, 
are not direct sums when $n\geq 3$. 
\begin{corollary}
[Determinant-like polynomials are not direct sums] 
\label{C:determinant}
Let $n \geq 3$. Suppose $k\neq \mathbb{F}_2$. 
Suppose $S=k[x_{i,j}]_{i,j=1}^{n}$ and $f=\sum_{\sigma \in S_n} a_{\sigma} x_{1,\sigma(1)} \cdots x_{n,\sigma(n)}$,
where $a_{\sigma}\in k^*$. Then $f$ is not a direct sum.  
\end{corollary}
%\begin{proof} 
%It is easy 
%to see that $f$ satisfies all conditions %\eqref{E:partial-states-non-overlap} and \eqref{E:recoverable-state}
%of Theorem \ref{MT4}. 
%\end{proof}

\begin{corollary}[Pfaffian-like polynomials are not direct sums] 
\label{C:pfaffian}
Let $n \geq 3$. Suppose $k\neq \mathbb{F}_2$. 
Suppose $S=k[x_{i,j}]_{1\leq i < j\leq 2n}$, 
where we set $x_{j,i}:=x_{i,j}$ for $j>i$, and \[
f=\sum_{\sigma \in S_{2n}} a_{\sigma} x_{\sigma(1),\sigma(2)} \cdots x_{
\sigma(2n-1),\sigma(2n)},\]
where $a_{\sigma}\in k^*$.  Then $f$ is not a direct sum.  
\end{corollary}
\begin{proof}[Proof of both corollaries]
It is easy 
to see that $f$ satisfies all conditions 
of Theorem \ref{MT4}. 
\end{proof}

Corollaries \ref{C:determinant} and \ref{C:pfaffian} are generalizations of \cite[Corollary 1.2]{apolarity}, whose proof 
relies on a theorem of Shafiei \cite{shafiei} saying that the apolar ideals of the generic determinant and permanent are generated in degree $2$; 
our approach is independent of Shafiei's results.

\begin{proof}[Proof of Theorem \ref{MT4}]
If $\chark(k)=p$, we 
 set \[
 \Xi_p:=\{(d_1,\dots,d_n) \mid \text{ $p$ divides $d_i$ for all $1\leq i \leq n$}\}
 \] to be the set of all monomials whose gradient point is trivial.

Suppose $f$ is a direct sum. Note that Assumption \eqref{E:partial-nonzero}
implies that $\dim_k \nabla f=n$. Then by Lemma \ref{L:gradient-ds}(3), and the assumption 
that $k\setminus \{0,1\} \neq \varnothing$,
there exists a form $g$ such that $\nabla g=\nabla f$ and $\nabla(g-cf)\neq 0$
for all $c\in k$.
Since $\nabla g=\nabla f$, 
then by Assumption \eqref{E:recoverable-state}, we must have 
$\Xi(g) \subset \Xi(f) \cup \Xi_p$. Then $\Xi(\partial g/\partial x_i) \subset
\Xi(\partial f/\partial x_i)$.

 Since $\partial g/\partial x_i \in \langle \nabla f\rangle$, Assumption \eqref{E:partial-states-non-overlap}
implies that in fact 
\[
\frac{\partial g}{\partial x_i} = c_i \frac{\partial f}{\partial x_i}, \ \text{for some $c_i\in k$}.
\]
Comparing the second partials, and using Assumption \eqref{E:connected} 
we conclude that $c_i=c_j$ for all $1\leq i< j\leq n$. 
We obtain $\nabla (g-c_1f)=(0)$, which is a contradiction.
\end{proof}

\section{Finding a balanced direct product decomposition algorithmically}
\label{S:algorithm}
In this section, we show how Theorem \ref{MT1} reduces the problem of 
finding a direct sum decomposition of a given smooth form\footnote{We recall that 
we call a homogeneous form $f\in k[x_1,\dots, x_n]$ smooth if it defines a smooth hypersurface
in $\PP^{n-1}$.} $f$ to a polynomial factorization problem.
To begin, suppose that we are given a smooth form $f \in \Sym^{d+1} V$ in some basis % $x_1,\dots, x_n$ of 
of $V$.  Then the associated form $A(f)$ is computed in the dual basis of $V^{\vee}$ as the form apolar to the Jacobian ideal $J_f$.  In fact, since we know that $A(f)$ has degree $n(d-1)$,
we reduce to a linear-algebraic problem of finding  the unique up to a scalar form of degree $n(d-1)$ apolar to the space $(J_f)_{n(d-1)}$.

To apply Theorem \ref{MT1}, we now need to determine if
 $A(f)\in \Sym^{n(d-1)} V^{\vee}$ decomposes as a balanced direct product, and if it does, then in what basis of $V^{\vee}$.  The following 
simple lemma explains how to do it (cf. \S\ref{S:notation} for the definition of the space of essential variables): %determine whether $A(f)$ is a balanced direct product and how to compute
%the balanced direct product decomposition, and the corresponding basis of $V^{\vee}$, when it exists.
\begin{lemma}\label{L:balanced-product} 
Suppose $\chark(k)\nmid (n(d-1))!$. 
For a smooth $f\in S_{d+1}$, the associated form $A(f)$ is a balanced direct product if and only if there is a 
non-trivial factorization $A(f)=G_1G_2$ such that 
\begin{equation}\label{direct-sum-A}
V=\bigl(G_1^{\perp}\bigr)_1+ \bigl(G_2^{\perp}\bigr)_1, \ \text{or, equivalently,}\quad E(G_1)\cap E(G_2)=(0) \subset V^{\vee}.
\end{equation}
Moreover, in this case, we have $\bigl(G_1^{\perp}\bigr)_1 \cap 
\bigl(G_2^{\perp}\bigr)_1=(0) \subset V$ and
$A(f)$ decomposes as a balanced direct product in any basis of $V^{\vee}$ such that its dual basis is compatible with the direct sum decomposition in 
Equation \eqref{direct-sum-A}.
\end{lemma}

\begin{proof} The equivalence of the two conditions in \eqref{direct-sum-A} follows from the 
fact that $E(G_i)\subset V^{\vee}$ is dual to $(G_i^{\perp})_1\subset V$.
The claim now follows from definitions and Theorem \ref{MT1} 
by observing that
for any non-trivial factorization $A(f)=G_1G_2$, we have
$
\bigl(G_1^{\perp}\bigr)_1 \cap \bigl(G_2^{\perp}\bigr)_1 \subset \bigl(A(f)^{\perp}\bigr)_1=(J_f)_1=(0)
$.
\end{proof}

\subsection{An algorithm for direct sum decompositions}  
Suppose $k$ is a field, with either 
$\chark(k)=0$ or $\chark(k)>\max\{n(d-1),d+1\}$, for
which there exists a polynomial factorization algorithm. Let $f\in k[x_1,\dots,x_n]_{d+1}$,
where $d\geq 2$. 

\vspace{-0.5pc}
\subsubsection*{Step 1:} Compute $J_f=(\partial f/\partial x_1,\dots, \partial f/\partial x_n)$ 
up to degree $n(d-1)+1$. If \[
(J_f)_{n(d-1)+1} \neq k[x_1,\dots, x_n]_{n(d-1)+1},
\] then 
$f$ is not smooth and we stop; otherwise, continue. 

\vspace{-0.5pc}

\subsubsection*{Step 2:} Compute $A(f)$ as the dual to $(J_f)_{n(d-1)}$:
\[
\langle A(f) \rangle = \left\{ T\in k[z_1,\dots, z_n]_{n(d-1)} \mid g\circ T = 0,  \ \text{for all $g\in (J_f)_{n(d-1)}$}\right\}.
\]
This can be done as follows: 

Compute the degree $n(d-1)$ part of the Gr\"obner basis of the Jacobian ideal
$J_f$, say, using the graded reverse lexicographic order. Suppose
$\{m_j\}_{j=1}^N$, where $N=\dim_k k[x_1,\dots, x_n]_{n(d-1)}$, are the monomials
in $k[x_1,\dots, x_n]_{n(d-1)}$ in the given order. Then
the output will be of the form 
\[
(J_f)_{n(d-1)}=(m_1-c_1m_i,\dots ,m_{i-1}-c_{i-1} m_i, m_{i+1},\dots ,m_N),
\]
where $m_i$ is the unique monomial of degree $n(d-1)$ which is not an initial monomial
of an element of $J_{n(d-1)}$. 
Let $\{\widehat{m}_j\}_{j=1}^N$ be the dual monomials in $k[z_1,\dots, z_n]$. Namely,
\[
m_i \circ \widehat{m}_j= \begin{cases} 1,& i=j, \\ 0, & i\neq j\end{cases}.
\]
For example, 
\[
\widehat{x_1^{d_1}\cdots x_n^{d_n}}=\frac{1}{(d_1)!\cdots (d_n)!} z_1^{d_1}\cdots z_n^{d_n}.
\]
Then 
\[
A(f)=\widehat{m_i}+\sum_{j<i} c_j \widehat{m_j}.
\]

\vspace{-0.5pc}

\subsubsection*{Step 3:} Compute the irreducible factorization of $A(f)$ in $k[z_1,\dots, z_n]$ 
and check for the existence of balanced direct product
factorizations using Lemma \ref{L:balanced-product}. If any exist, then $f$ is a direct sum; otherwise,
$f$ is not a direct sum. 

\vspace{-0.5pc}

\subsubsection*{Step 4:} For every balanced direct product factorization of $A(f)$,
Lemma \ref{L:balanced-product}  
gives a basis of $V$ in which $f$ decomposes as a direct sum.

\smallskip 

The above algorithm was implemented in a Macaulay 2 package written by Justin Kim and Zihao Fang (its source code is available upon request). 
In what follows, we give a few examples
of the algorithm in action.
\begin{remark} Jaros{\l}aw Buczy\'nski has pointed out that already \emph{Step 2} in the above algorithm is 
computationally highly expensive when $n$ and $d$ are large. However, it is reasonably fast when 
both $n$ and $d$ are small, with Example \ref{E:random} below taking only a few seconds.
\end{remark}

\begin{example}\label{E:intro}
Consider $f={x}_{1}^{3}+3 {x}_{1}^{2} {x}_{2}+3 {x}_{1} {x}_{2}^{2}+2 {x}_{2}^{3}+3
      {x}_{1}^{2} {x}_{3}+6 {x}_{1} {x}_{2} {x}_{3}+4 {x}_{2}^{2} {x}_{3}+3
      {x}_{1} {x}_{3}^{2}+4 {x}_{2} {x}_{3}^{2}+2 {x}_{3}^{3}$
      in $\QQ[x_1,x_2,x_3]$ from the introduction. Its discriminant is nonzero and so we
      can compute the associated form of $f$.
      We have
\begin{multline*}
      A(f)=-{z}_{1}^{3}+{z}_{1}^{2} {z}_{2}+\frac{1}{2} {z}_{1} {z}_{2}^{2}+{z}_{1}^{2}
       {z}_{3}-2 {z}_{1} {z}_{2} {z}_{3}+\frac{1}{2}{z}_{1} {z}_{3}^{2}\\=
       \frac{1}{2} z_1 
       (-{z}_{1}^{2}+{z}_{1} {z}_{2}+\frac{1}{2} {z}_{2}^{2}+{z}_{1} {z}_{3}-2 {z}_{2} {z}_{3}+\frac{1}{2} {z}_{3}^{2}).%= \frac{1}{4} z_1 ((z_1+z_2+z_3)^2-4z_1(z_1+z_2+z_3)+z_2^2)
\end{multline*}
The first factor is a polynomial in $z_1$ and the second factor is a polynomial in $z_1-z_2$ 
and $z_1-z_3$. It follows that $A(f)$ is a balanced direct product and so $f$ is a direct sum.
Indeed, the basis of $V$ dual to the basis $\{z_1, z_2-z_1, z_3-z_1\}$ of $V^{\vee}$ is precisely
$\{x_1+x_2+x_3, x_2, x_3\}$. In this basis, 
the original polynomial becomes a direct sum:
\[
f=(x_1+x_2+x_3)^{3}+{x}_{2}^{3}+{x}_{2}^{2} {x}_{3}+{x}_{2} {x}_{3}^{2}+{x}_{3}^{3}.
\]
\end{example}

\begin{example}[Binary quartics]
\label{E:non-closed}
Suppose $k$ is an algebraically closed field of characteristic $0$. Then every smooth binary quartic has a standard form
\begin{equation*}
f_t=x_1^4+x_2^4+tx_1^2 x_2^2, \quad t\neq \pm 2.
\end{equation*}
Up to a scalar, the associated form of $F_t$ is  
\[
A(f_t)=t(z_1^4+z_2^4)-12z_1^2z_2^2.
\]
Clearly, $A(f_t)$ is singular if and only if $t=0$, or $t=\pm 6$. For these values of $t$,
$A(f_t)$ is in fact a balanced direct product, and so $f_t$ is a direct sum. Namely, up to scalars, we have:
\begin{align*}
A(f_0) \,&=z_1^2z_2^2, && f_0 =x_1^4+x_2^4, \\
A(f_6) \,&=(z_1^2-z_2^2)^2=(z_1-z_2)^2(z_1+z_2)^2, && f_6 =(x_1-x_2)^4+(x_1+x_2)^4, \\
A(f_{-6}) &=(z_1^2+z_2^2)^2=(z_1-iz_2)^2(z_1+iz_2)^2, && f_{-6} =(x_1-ix_2)^4+(x_1+ix_2)^4.
\end{align*}
Note that over $\RR$, the associated form $A(f_{-6})=(z_1^2+z_2^2)^2$ is not a balanced direct product. Hence $f_{-6}$ is not a direct sum over $\RR$ by Theorem \ref{MT1}.  
Since the apolar ideal of $f_{-6}$ is the same over $\RR$ and over $\CC$, this example illustrates that the direct sum 
decomposability criterion of \cite{apolarity} fails over non-closed fields.
\end{example}

\begin{example}\label{E:random} Consider the following element in $\QQ[x_1,x_2,x_3,x_4]_{4}$:
{\footnotesize
\begin{multline*}
f={x}_{1}^{4}+4 {x}_{1}^{3} {x}_{2}+6 {x}_{1}^{2} {x}_{2}^{2}+4 {x}_{1} {x}_{2}^{3}+2 {x}_{2}^{4}+8 {x}_{1}^{3} {x}_{3}+24 {x}_{1}^{2} {x}_{2} {x}_{3}+24 {x}_{1}
      {x}_{2}^{2} {x}_{3}+8 {x}_{2}^{3} {x}_{3}+24 {x}_{1}^{2} {x}_{3}^{2} \\ +48 {x}_{1} {x}_{2} {x}_{3}^{2}+24 {x}_{2}^{2} {x}_{3}^{2}+32 {x}_{1} {x}_{3}^{3}+32 {x}_{2}
      {x}_{3}^{3}+17 {x}_{3}^{4}-12 {x}_{1}^{3} {x}_{4}-36 {x}_{1}^{2} {x}_{2} {x}_{4}-36 {x}_{1} {x}_{2}^{2} {x}_{4}-12 {x}_{2}^{3} {x}_{4} \\ -72 {x}_{1}^{2} {x}_{3}
      {x}_{4}-144 {x}_{1} {x}_{2} {x}_{3} {x}_{4}  -72 {x}_{2}^{2} {x}_{3} {x}_{4}-144 {x}_{1} {x}_{3}^{2} {x}_{4}-144 {x}_{2} {x}_{3}^{2} {x}_{4}-96 {x}_{3}^{3}
      {x}_{4}+54 {x}_{1}^{2} {x}_{4}^{2} +108 {x}_{1} {x}_{2} {x}_{4}^{2}\\ +54 {x}_{2}^{2} {x}_{4}^{2}+216 {x}_{1} {x}_{3} {x}_{4}^{2}+217 {x}_{2} {x}_{3} {x}_{4}^{2}+216
      {x}_{3}^{2} {x}_{4}^{2}-108 {x}_{1} {x}_{4}^{3}-108 {x}_{2} {x}_{4}^{3}-216 {x}_{3} {x}_{4}^{3}+82 {x}_{4}^{4}.
      \end{multline*}
      }
      Then its associated form is
{\footnotesize
\begin{multline*}
     A(f)= 9785 {z}_{1}^{8}-32316 {z}_{1}^{7} {z}_{2}+26370 {z}_{1}^{6} {z}_{2}^{2}-260 {z}_{1}^{5} {z}_{2}^{3}+15 {z}_{1}^{4} {z}_{2}^{4}+24 {z}_{1}^{3} {z}_{2}^{5}-19488
      {z}_{1}^{7} {z}_{3}+40920 {z}_{1}^{6} {z}_{2} {z}_{3} \\ -25620 {z}_{1}^{5} {z}_{2}^{2} {z}_{3}-180 {z}_{1}^{4} {z}_{2}^{3} {z}_{3}+60 {z}_{1}^{3} {z}_{2}^{4}
      {z}_{3}-12 {z}_{1}^{2} {z}_{2}^{5} {z}_{3}+8730 {z}_{1}^{6} {z}_{3}^{2}-11910 {z}_{1}^{5} {z}_{2} {z}_{3}^{2}+6390 {z}_{1}^{4} {z}_{2}^{2} {z}_{3}^{2} \\+30
      {z}_{1}^{3} {z}_{2}^{3} {z}_{3}^{2}-595 {z}_{1}^{5} {z}_{3}^{3}-495 {z}_{1}^{4} {z}_{2} {z}_{3}^{3}+15 {z}_{1}^{3} {z}_{2}^{2} {z}_{3}^{3}-5 {z}_{1}^{2}
      {z}_{2}^{3} {z}_{3}^{3}+15 {z}_{1}^{4} {z}_{3}^{4}+120 {z}_{1}^{3} {z}_{2} {z}_{3}^{4}+12 {z}_{1}^{3} {z}_{3}^{5}-12 {z}_{1}^{2} {z}_{2} {z}_{3}^{5} \\-4194
      {z}_{1}^{7} {z}_{4}-9000 {z}_{1}^{6} {z}_{2} {z}_{4}+17820 {z}_{1}^{5} {z}_{2}^{2} {z}_{4}-360 {z}_{1}^{4} {z}_{2}^{3} {z}_{4}+90 {z}_{1}^{3} {z}_{2}^{4}
      {z}_{4}-7200 {z}_{1}^{6} {z}_{3} {z}_{4}+21600 {z}_{1}^{5} {z}_{2} {z}_{3} {z}_{4}\\-17280 {z}_{1}^{4} {z}_{2}^{2} {z}_{3} {z}_{4}+6480 {z}_{1}^{5} {z}_{3}^{2}
      {z}_{4}-8640 {z}_{1}^{4} {z}_{2} {z}_{3}^{2} {z}_{4}+4320 {z}_{1}^{3} {z}_{2}^{2} {z}_{3}^{2} {z}_{4}-720 {z}_{1}^{4} {z}_{3}^{3} {z}_{4}+90 {z}_{1}^{3}
      {z}_{3}^{4} {z}_{4}-7395 {z}_{1}^{6} {z}_{4}^{2}\\+7140 {z}_{1}^{5} {z}_{2} {z}_{4}^{2}+2970 {z}_{1}^{4} {z}_{2}^{2} {z}_{4}^{2}-60 {z}_{1}^{3} {z}_{2}^{3}
      {z}_{4}^{2}+15 {z}_{1}^{2} {z}_{2}^{4} {z}_{4}^{2}+3120 {z}_{1}^{5} {z}_{3} {z}_{4}^{2}-720 {z}_{1}^{4} {z}_{2} {z}_{3} {z}_{4}^{2}-2880 {z}_{1}^{3} {z}_{2}^{2}
      {z}_{3} {z}_{4}^{2}\\+1080 {z}_{1}^{4} {z}_{3}^{2} {z}_{4}^{2}-1440 {z}_{1}^{3} {z}_{2} {z}_{3}^{2} {z}_{4}^{2}+720 {z}_{1}^{2} {z}_{2}^{2} {z}_{3}^{2}
      {z}_{4}^{2}-120 {z}_{1}^{3} {z}_{3}^{3} {z}_{4}^{2}+15 {z}_{1}^{2} {z}_{3}^{4} {z}_{4}^{2}-1800 {z}_{1}^{5} {z}_{4}^{3}+2880 {z}_{1}^{4} {z}_{2} {z}_{4}^{3}\\+1440
      {z}_{1}^{4} {z}_{3} {z}_{4}^{3}-1440 {z}_{1}^{3} {z}_{2} {z}_{3} {z}_{4}^{3}+30 {z}_{1}^{4} {z}_{4}^{4}+240 {z}_{1}^{3} {z}_{2} {z}_{4}^{4}+120 {z}_{1}^{3}
      {z}_{3} {z}_{4}^{4}-120 {z}_{1}^{2} {z}_{2} {z}_{3} {z}_{4}^{4}+36 {z}_{1}^{3} {z}_{4}^{5}+2 {z}_{1}^{2} {z}_{4}^{6}.    
      \end{multline*}}
One checks that $A(f)=G_1G_2$, where $G_1=z_1^2$ with $E(G_1)=z_1$,
and $E(G_2)=\langle   3z_3+2z_4, 3z_2+z_4, 3z_1+z_4\rangle$, is a balanced direct product factorization of $A(f)$. It follows
that $f$ is a direct sum. In fact, $f$ is projectively equivalent to 
$
      x_1^4+(x_2^4+x_3^4+x_4^4+x_2x_3x_4^2).
$
\end{example}

\vspace{-0.5pc}

\section*{Acknowledgments} The author is grateful to Jarod Alper for an introduction to the subject, 
Alexander Isaev for numerous stimulating discussions that inspired this work, and Zach Teitler for 
questions that motivated most of the results in Section \ref{S:necessary}. We thank 
Zhenjian Wang  for alerting us to his earlier work \cite{wang} related to Section \ref{S:injectivity} and interesting questions. 

The author was partially supported by  
the NSA Young Investigator grant H98230-16-1-0061 and Alfred P. Sloan Research Fellowship.
Justin Kim and Zihao Fang wrote a Macaulay 2 package for computing associated forms while supported by the Boston College Undergraduate Research Fellowship grant under the direction of the author.  

\bibliographystyle{plain}
\bibliography{associated}{}
\end{document}